\documentclass{endm}
\usepackage{endmmacro}
\usepackage{graphicx}
\usepackage{amssymb}

\usepackage{tikz}
\usepackage{tikz-cd}
\usetikzlibrary{arrows,automata}
\usepackage[utf8]{inputenc}
\usepackage[T1]{fontenc}
\usepackage[all]{xy}
\usepackage{amsmath,amssymb}

\newcommand\rsmraise[1]{%
  \ifx#1\displaystyle .8\else
    \ifx#1\textstyle .8\else
      \ifx#1\scriptstyle .6\else
        .45%
      \fi
    \fi
  \fi}
\usepackage{mathptmx}       
\usepackage{helvet}         
\usepackage{courier}        
\usepackage{type1cm}        
\usepackage{xhfill}         
\usepackage{xpatch}
\usepackage{textcase}
\makeatletter
\xpatchcmd{\@sect}{\uppercase}{\MakeTextUppercase}{}{}
\xpatchcmd{\@sect}{\uppercase}{\MakeTextUppercase}{}{}
\makeatother
\usepackage{makeidx}         
\usepackage{graphicx}        
\usepackage{multicol}        
\usepackage{pgf,tikz}
\usetikzlibrary{arrows}
\usepackage{ulem}

\def\dd{\mathbf{d}}

\def\scal#1#2{\langle #1\bv#2 \rangle}
\def\sscal#1#2{\langle #1\| #2 \rangle}
\def\pointir{\unskip . --- \ignorespaces} 
\def\ncp#1#2{#1\langle #2\rangle}
\def\ncs#1#2{#1\langle \!\langle #2\rangle \!\rangle}
\def\dext#1#2{#1\{ \!\{ #2\} \!\}}
\usepackage{textcomp}
\usepackage{enumerate}
\usepackage{tikz}
\usetikzlibrary{arrows,decorations.pathmorphing}
\usetikzlibrary{backgrounds,positioning,fit,petri}
\usetikzlibrary{matrix}
\def\calE{{\mathcal E}}

\def\calX{{\cal X}}

\def\pol#1{\langle #1 \rangle}

\def\QY{\ncp{\Q}{Y}}

\def\scal#1#2{\langle #1\bv#2 \rangle}
\def\bv{\mid}

\def\abs#1{\bv\!#1\!\bv}

\def\ncp#1#2{#1\langle #2\rangle}

\def\conc{{\tt conc}}
\def\trr{\triangleright}
\def\trl{\triangleleft}
\def\span{\mathop\mathrm{span}\nolimits}
\def\Dom{\mathrm{Dom}}

\newcounter{per1}
\begingroup
\def\2#1{\ifnum#1<10 0\fi\the#1}
\xdef\isodayandtime{
{\2\day-\2\month-\the\year\space\2{\count0}:%
\2{\count2}}}
\endgroup
\newcommand{\bi}{\begin{itemize}}
\newcommand{\ei}{\end{itemize}}
\newcommand{\bd}{\begin{description}}
\newcommand{\ed}{\end{description}}

\usepackage{amsmath}
\usepackage{amsfonts}
\usepackage{amssymb}

\newcommand{\calA}{{\mathcal A}}

\newcommand{\calL}{{\mathcal L}}
\newcommand{\calM}{{\mathcal M}}
\newcommand{\calH}{{\mathcal H}}
\newcommand{\calC}{{\mathcal C}}

\newcommand{\calD}{{\mathcal D}}

\newcommand{\N}{{\mathbb N}}
\renewcommand{\P}{\N_{>0}}
\newcommand{\Z}{{\mathbb Z}}
\newcommand{\Q}{{\mathbb Q}}

\newcommand{\C}{{\mathbb C}}

\newcommand{\Frac}[2]{\displaystyle \frac{#1}{#2}}
\newcommand{\Sum}[2]{\displaystyle{\sum_{#1}^{#2}}}
\newcommand{\Prod}[2]{\displaystyle{\prod_{#1}^{#2}}}

\def\pol#1{\langle #1 \rangle}
\def\Lyn{{\mathcal Lyn}}

\def\CX{\C \langle X \rangle}
\def\CY{\C \langle Y \rangle}
\def\abs#1{|#1|}
\def\absv#1{\|#1\|}

 \def\shuffle{\mathop{_{^{\sqcup\!\sqcup}}}} 

\gdef\stuffle{\;%
  \setlength{\unitlength}{0.0125cm}%
  \begin{picture}(20,10)(220,580) 
  \thinlines 
  \put(220,592){\line( 0,-1){ 10}} 
  \put(220,582){\line( 1, 0){ 20}} 
  \put(240,582){\line( 0, 1){ 10}} 
  \put(230,592){\line( 0,-1){ 10}} 
  \put(225,587){\line( 1, 0){ 10}} 
  \end{picture}\; 
}

\newcommand{\Li}{\operatorname{Li}}

\def\L{\mathrm{L}}
\def\H{\mathrm{H}}

\def\P{\mathrm{P}}

\def\deg{\mathrm{deg}}

\newcommand{\poly}[2]{#1 \langle #2 \rangle}

\def\CX{\poly{\C}{X}}

\def\CY{\poly{\C}{Y}}

\newcommand{\serie}[2]{#1 \langle \! \langle #2 \rangle \! \rangle}

\def\QY_0{\Q\left\langle{Y_0}\right\rangle}
\def\pol#1{\langle #1 \rangle}

\def\Lie{{\cal L}ie}

\def\CX{\C \langle X \rangle}

\def\Sum{\displaystyle\sum}
\def\Prod{\displaystyle\prod}
\def\Frac{\displaystyle\frac}
\def\path{\rightsquigarrow}

\def\bv{\mid}

\def\abs#1{\bv\!#1\!\bv}

\def\rd{\triangleright}
\def\rg{\triangleleft}
\def\deg{\mathop\mathrm{deg}\nolimits}

\def\supp{\mathop\mathrm{supp}\nolimits}

\def\binom#1#2{{#1\choose#2}}

\usepackage{color}
\usepackage{ulem}

\def\scal#1#2{\langle #1\bv#2 \rangle}
\def\ncp#1#2{#1\langle #2\rangle}
\def\ncs#1#2{#1\langle \!\langle #2\rangle \!\rangle}

\begin{document}
\begin{frontmatter}
\vskip-1cm
\title{Towards a Picard-Vessiot theory of noncommutative differential equations}
\author{G.H.E. Duchamp}
\address{University Paris 13, Sorbonne Paris City, 93430 Villetaneuse, France,}
\author{V. Hoang Ngoc Minh}
\address{University of Lille, 1 Place Déliot, 59024 Lille, France,}
\author{V. Nguyen Dinh}
\address{LIPN-UMR CNRS 7030, 99, avenue Jean-Baptiste Clément 93430 Villetaneuse, France,}
\author{P. Simonnet}
\address{University of Corsica, 20250 Corte, France.}

\begin{abstract}
A Chen generating series, along a path and with respect to $m$ differential forms,
is a noncommutative series on $m$ letters, with coefficients which are holomorphic functions
over a simply connected manifold. In other words, it is a series with variable (holomorphic) coefficients.
Such a series satisfies a first order noncommutative differential equation which is considered,
by some authors, as the universal differential equation. In this case, universality
can be seen by replacing each letter by constant matrices (resp. holomorphic vector fields)
and then solving a system of linear (resp. nonlinear) differential equations.

Via rational series, on noncommutative indeterminates and with coefficients in rings,
and their non-trivial combinatorial Hopf algebras (or pseudo-Hopf algebras), 
we propose a first step of a noncommutative
Picard-Vessiot theory and we illustrate it with the case of linear differential equations
with singular regular singularities thanks to the universal equation previously mentioned.
\end{abstract}
\end{frontmatter}
\vfill

\tableofcontents
\section{Introduction}\label{introduction}
Combinatorial Picard-Vessiot (PV for short) theory of bilinear systems\footnote{Namely - locally - linear
of the states $q_1,\ldots,q_N$ and linear of the inputs $u_0,\ldots,u_m$ \cite{PV}.}
was realized by Fliess and Reutenauer \cite{PV}, as an application of differential 
algebra \cite{Kolchin,Ritt}. This theory allows to employ, with success, linear algebraic groups 
in control theory 
(\textit{i.e.} as symmetry groups of linear differential equations), 
for which some questions were solved thanks to the theory of Hopf algebras 
\cite{Cartier} and some combinatorial and effective aspects were set in \cite{reutenauer}.

Let us, for instance, consider the following nonlinear dynamical system
\begin{eqnarray}\label{nonlinear}
\dot q(z)=A_0(q)u_0(z)+\ldots+A_m(q)u_m(z),&q(z_0)=q_0,&y(z)=f(q(z)),
\end{eqnarray}
where
\begin{enumerate}
\item $y$ is the output,
\item the vector state $q=(q_1,\ldots,q_n)$ belongs to a complex holomorphic manifold $\calM$ of dimension $n$,
\item the observation $f$ is defined within a fixed connected neighbourhood\footnote{In this
introductive description the points are loosely identified with their coordinates through some
chart $\varphi_U:U\to \C^n$ likewise, in \cite{reutenauerrealisation}, the space of holomorphic
functions $\calH(U)$ is described by $\C^{\rm cv}[\![q_1,\ldots,q_n]\!]$.} $U$ of of the initial state $q_0$.
\item the vector fields $(A_i)_{i=0,\ldots,m}$ are defined with respect to the coordinates as follows
\begin{eqnarray}\label{vectorfield}
A_i=\sum_{j=1}^nA_i^j(q)\frac{\partial}{\partial q_j},&\mbox{with}&A_i^j(q)\in\calH(U),
\end{eqnarray}
\item the inputs $(u_i)_{i=0,\ldots,m}$, as well as their inverses $(u_i^{-1})_{i=0,\ldots,m}$,
belong to a subring, $\calC_0$, of the ring of holomorphic functions $\calH(\Omega)$ with
neutral element $1_{\calH(\Omega)}$ over the simply connected manifold\footnote{This (usually
one dimensional) manifold will be the support of the iterated integrals below.} $\Omega$. 
\end{enumerate}

It is convenient (and possible) to separate the contribution of the vector fields $(A_i)_{i=0,\ldots,m}$ and that
of the differential forms $(\omega_i)_{i=0,\ldots,m}$, defined by the inputs $\omega_i(z)=u_i(z)dz$,
through the encoding alphabet $X=\{x_i\}_{i=0,\ldots,m}$ which generates the monoid $X^*$ with
neutral element $1_{X^*}$. Indeed, the output $y$ can be computed by
\begin{eqnarray}\label{output}
y(z_0,z)=\sscal{C_{z_0\path z}}{\sigma f_{|_{q_0}}}=\sum_{w\in X^*}\alpha_{z_0}^z(w){\cal Y}(w)[f]_{|_{q_0}},
\end{eqnarray}
as the pairing (under suitable convergence conditions \cite{fliess2,words03,cade,reutenauerrealisation})
between the Chen series\footnote{By a Ree's theorem \cite{ree}, there is a primitive series
$L_{z_0\path z}=\sum_{n\ge1}L_n\in\widehat{{\calH(\Omega)}\pol{X}}$ s.t. $e^{L_{z_0\path z}}=C_{z_0\path z}$,
meaning that $C_{z_0\path z}$ is group-like and $L_n$ is (homogenous of degree $n\ge1$) primitive series.}
of $(\omega_i)_{i=0,\ldots,m}$ along the path $z_0\path z$ over $\Omega$, $C_{z_0\path z}\in\serie{\calH(\Omega)}{X}$
\cite{cartier1}, and the generating series of \eqref{nonlinear}, $\sigma f_{|_{q_0}}\in\serie{\calH(U)}{X}$
\cite{fliess2}, defined as follows
\begin{eqnarray}\label{Chen}
C_{z_0\path z}:=\Sum_{w\in X^*}\alpha_{z_0}^z(w)w&\mbox{and}&\sigma f_{|_{q_0}}:=\sum_{w\in X^*}{\cal Y}(w)[f]_{|_{q_0}}w,
\end{eqnarray}
where, in \eqref{output}--\eqref{Chen}, the iterated integral $\alpha_{z_0}^z(w)$ and the differential operator ${\cal Y}(w)$,
are decoded, from the word $w\in X^*$, recursively as follows
\begin{eqnarray}\label{notation}
\left\{\begin{array}{lcll}
\alpha_{z_0}^z(w)=1_{\calH(\Omega)}&\mbox{and}&{\cal Y}(w)=\mathrm{Id},
&\mbox{for }w=1_{X^*},\\
\alpha_{z_0}^z(w)=\displaystyle\int_{z_0}^z\omega_i(s)\alpha_{z_0}^{s}(v)&\mbox{and}&{\cal Y}(w)=A_i\circ{\cal Y}(v),
&{\displaystyle\mbox{for }w=x_iv,\hfill\atop\displaystyle x_i\in X,v\in X^*.}
\end{array}\right.
\end{eqnarray}

In this work, following this route, considering the differential ring $(\calH(\Omega),\partial)$ and
equipping $\serie{\calH(\Omega)}{X}$ with the derivation defined, for any $S\in\ncs{\calH(\Omega)}{X}$, by
\begin{eqnarray}
{\bf d}S=\sum_{w\in X^*}(\partial\scal{S}{w})w,
\end{eqnarray}
we can see that the Chen series satisfies the following noncommutative differential equation
\begin{eqnarray}\label{NCDE}
{\bf d}S=MS&\mbox{with}&M=u_0x_0+\ldots+u_mx_m,
\end{eqnarray}
considered by many authors as the universal differential equation \cite{cartier1,deligne,drinfeld1,drinfeld2,CM}.
Universality can be seen by specialization, \textit{i.e.} replacing the letters by constant matrices
(resp. holomorphic vector fields) and therefore obtaining linear (resp. nonlinear) differential equations 
(see Remark \ref{rk3} below)
as well as their solutions. From equation \eqref{NCDE}, it follows (see, for example, \cite{eulerianfunctions}) that
a PV theory of nonlinear systems \eqref{nonlinear} should be intimately connected with \eqref{NCDE} (the reader may
remark that, due to the connectedness of $\Omega$, the constants of $(\ncs{\calH(\Omega)}{X},{\bf d})$ are
\begin{eqnarray}
\mathrm{Const}(\serie{\calH(\Omega)}{X})=\ker{\bf d}=\serie{\C.1_{\calH(\Omega)}}{X}.
\end{eqnarray}

This culminates with the fact that the coefficients of any suitable\footnote{\textit{i.e.} group-like
at one - interior or frontier - point.} solution is group-like, \textit{i.e.} satisfies\footnote{
In the first identity, also called Friedrichs criterion, is involved the shuffle product ($\shuffle$)
\cite{cartier1,fliess1,reutenauer}.}, for any $u,v\in X^*$ and $x_i\in X$,
\begin{eqnarray}
\partial\scal{S}{x_iu}=u_i\scal{S}{v}&\mbox{and}&\scal{S}{u\shuffle v}=\scal{S}{u}\scal{S}{v};
\scal{S}{1_{X^*}}=1_{\calH(\Omega)}
\end{eqnarray}
Due to the fact that $\Omega$ is simply connected, the coordinate values of this series
only depend on the endpoints and not on paths drawn on $\Omega$. Denoting the subalgebra
of $(\calH(\Omega),\partial)$ generated by the family $(f_i)_{i\in I}$ and derivatives
by $\dext{\C}{(f_i)_{i\in I}}$ \cite{VdP} (\textit{i.e.} the differential algebra
generated by $(f_i)_{i\in I}$), it follows that \cite{PV}
\begin{eqnarray}
\mathrm{span}_{\C}\{\scal{{\bf d}^lS}{w}\}_{w\in X^*,l\ge0}
&\subset&\mathrm{span}_{\dext{\C}{(u_i)_{i=0,..,m}}}\{\scal{S}{w}\}_{w\in X^*}\\
&\subset&\mathrm{span}_{\dext{\C}{(u_i^{\pm1})_{i=0,..,m}}}\{\scal{S}{w}\}_{w\in X^*}\label{der}
\end{eqnarray}
and then, in Section \ref{PVNC}, the isomorphism between $\mathrm{span}_{\dext{\C}{(u_i^{\pm1})_{i=0,..,m}}}\{\alpha_{z_0}^z(w)\}_{w\in X^*}$
and $\dext{\C}{(u_i^{\pm1})_{i=0,..,m}}\otimes_{\C}\mathrm{span}_{\C}\{\alpha_{z_0}^z(w)\}_{w\in X^*}$ will be examined (Theorem \ref{ind_lin})
via the PV-extension related to \eqref{NCDE} and, on the other hand, the output of \eqref{nonlinear} will be computed (Theorem \ref{PeanoBacker})
by pairing the series given in \eqref{Chen}. As example, this calculation will be achieved according to the algebraic combinatorics of rational
series, established beforehand in Sections \ref{Combinatorial} (Theorems \ref{isomorphy}, \ref{residual}) and \ref{TST}
(Theorems \ref{exchangeable}, \ref{Subbialgebras}).

\section{Combinatorial framework}\label{Combinatorial}
In this section, coefficients are taken in a commutative ring\footnote{although some of
the properties already hold for a general commutative semiring \cite{berstel}.} $A$ and,
unless explicitly stated, all tensor products will be considered over the ambient ring (or field). 

\subsection{Factorization in bialgebras}
In section \ref{introduction}, the encoding alphabet $X$ was already introduced.
In particular, for $m=1$ (\textit{i.e.} $X=\{x_0,x_1\})$,
let us note that there are one-to-one correspondences
\begin{eqnarray}
(s_1,\ldots,s_r)\in\N_+^r\leftrightarrow x_0^{{s_1}-1}x_1\ldots x_0^{{s_r}-1}x_1\in X^*x_1
\displaystyle\mathop{\rightleftharpoons}_{\pi_X}^{\pi_Y}y_{s_1}\ldots y_{s_r}\in Y^*,
\end{eqnarray}
where $Y:=\{y_k\}_{k\ge1}$ and $\pi_X$ is the $\conc$ morphism, from $\ncp{A}{Y}$ to $\ncp{A}{X}$, mapping
$y_k$ to $x_0^{k-1}x_1$. This morphism $\pi_X$ admits an adjoint $\pi_Y$ for the two standard scalar products\footnote{
That is to say $(\forall p\in\ncp{A}{X})\;(\forall q\in\ncp{A}{Y})\;(\scal{\pi_Yp}{q}_Y=\scal{p}{\pi_Xq}_X)$.}
which has a simple combinatorial description: the restriction of $\pi_Y$ to the subalgebra $(A1_{X^*}\oplus\ncp{A}{Y}x_1,\conc)$,
is an isomorphism given by $\pi_Y(x_0^{k-1}x_1)=y_k$ (and the kernel of the non-restricted $\pi_Y$ is $\ncp{A}{X}x_0$).
For all matters concerning finite ($X$ and similar) or infinite ($Y$ and similar) alphabets,
we will use a generic model noted $\calX$ in order to state their common combinatorial features.
Let us recall also that the coproduct $\Delta_{\conc}$ is defined, for any $w\in\calX^*$, as follows
\begin{eqnarray}\label{Dconc}
\Delta_{\conc}w=\sum_{u,v\in\calX^*,uv=w}u\otimes v.
\end{eqnarray}

As an algebra the $A$-module $\ncp{A}{\calX}$ is equipped with the associative unital
concatenation and the associative commutative and unital shuffle product. The latter being
defined, for any $x,y\in\calX$ and $u,v,w\in\calX^*$, by the following recursion
\begin{eqnarray}\label{shuffle}
w\shuffle 1_{\calX^*}=1_{\calX^*}\shuffle w=w&\mbox{and}&
xu\shuffle yv=x(u\shuffle yv)+y(xu\shuffle v)
\end{eqnarray}
or, equivalently, by its dual comultiplication (which is a morphism for concatenations\footnote{On $\ncp{A}{\calX}$
and $\ncp{A}{\calX}\otimes\ncp{A}{\calX}$, respectively.}),  defined, for each letter $x\in\calX$, by
\begin{eqnarray}\label{Dshuffle}
\Delta_{\shuffle}x=1_{\calX^*}\otimes x+x\otimes1_{\calX^*}.
\end{eqnarray}

Once $\calX$ has been totally ordered\footnote{For technical reasons, the orders $x_0<x_1$ (for $X$)
and $y_1>\ldots y_n>y_{n+1}>\ldots$ (for $Y$) are usual.}, the set of Lyndon words over $\calX$ will
be denoted by $\Lyn\calX$. A pair of Lyndon words $(l_1,l_2)$ is called the standard factorization
of a Lyndon $l$ (and will be noted $(l_1,l_2)=st(l)$) if $l=l_1l_2$ and $l_2$ is the longest nontrivial
proper right factor of $l$ or, equivalently, its smallest such (for the lexicographic ordering, see
\cite{lothaire}  for proofs and details). According to a theorem by Radford, the set of Lyndon words
form a pure transcendence basis of the $A$-shuffle algebras $(\ncp{A}{\calX},\shuffle,1_{\calX^*})$.

It is well known that the enveloping algebra $\mathcal{U}(\ncp{\calL ie_{A}}{\calX})$
is isomorphic to the (connected, graded and co-commutative) bialgebra\footnote{In case $A$
is a $\Q$-algebra, the isomorphism $\mathcal{U}(\ncp{\calL ie_{A}}{\calX})\simeq\calH_{\shuffle}(\calX)$
can also be seen as an easy application of the CQMM theorem.} $\calH_{\shuffle}(\calX)=(\ncp{A}{\calX},
\allowbreak\conc,1_{\calX^*},\Delta_{\shuffle},{\tt e})$ (the counit being here
${\tt e}(P)=\scal{P}{1_{\calX^*}}$) and, via the pairing
\begin{eqnarray}
\ncs{A}{\calX}\otimes_{A}\ncp{A}{\calX}&\longrightarrow&A,\\
T\otimes P&\longrightarrow&\scal{T}{P}:=\sum_{w\in\calX^*}\scal{T}{w}\scal{P}{w},\label{pairing}
\end{eqnarray}
we can, classically, endow $\ncp{A}{\calX}$ with the graded\footnote{For ${\cal X}=X$ or $=Y$ the 
corresponding monoids are equipped with length functions, for $X$ we consider the length of words
and for $Y$ the length is given by the weight $\ell(y_{i_1}\ldots y_{i_n})=i_1+\ldots+i_n$.
This naturally induces a grading of $\ncp{A}{{\cal X}}$ and $\ncp{\calL ie_{A}}{{\cal X}}$ in free
modules of finite dimensions. For general  ${\cal X}$, we consider the fine grading \cite{reutenauer}
\textit{i.e.} the grading by all partial degrees which, as well, induces a grading of $\ncp{A}{{\cal X}}$
and $\ncp{\calL ie_{A}}{{\cal X}}$ in free modules of finite dimensions.} linear basis
$\{P_w\}_{w\in \calX^*}$ (expanded after any homogeneous basis $\{P_l\}_{l\in \Lyn\calX}$
of $\ncp{\calL ie_{A}}{\calX}$) and its graded dual basis $\{S_w\}_{w\in\calX^*}$ 
(containing the pure transcendence basis $\{S_l\}_{l\in\Lyn\calX}$ of the $A$-shuffle algebra).
In the case when $A$ is a $\Q$-algebra, we also have the following factorization\footnote{
Also called MSR factorization after the names of M\'elan\c con,
Sch\"utzenberger and Reutenauer.} of the diagonal series,
\textit{i.e.} \cite{reutenauer} (here all tensor products are over $A$)
\begin{eqnarray}\label{diagonalX}
{\cal D}_{\calX}:=\sum_{w\in\calX^*}w\otimes w=\sum_{w\in\calX^*}S_w\otimes P_w
=\prod_{l\in\Lyn\calX}^{\searrow}e^{S_l\otimes P_l}
\end{eqnarray}
and (still in case $A$ is a $\Q$-algebra) dual bases of homogenous 
polynomials $\{P_w\}_{w\in\calX^*}$
and $\{S_w\}_{w\in\calX^*}$ can be constructed recursively as follows
\begin{eqnarray}\label{quarante-deux}
\left\{\begin{array}{llll}
P_x=x,
&S_x=x
&\mbox{for }x\in\calX,\\
P_l=[P_{l_1},P_{l_2}],
&S_l=yS_{l'},
&{\displaystyle\mbox{for }l=yl'\in\Lyn\calX - \calX\atop\displaystyle st(l)=(l_1,l_2),}\\
P_w=P_{l_1}^{i_1}\ldots P_{l_k}^{i_k},
&S_w=\displaystyle\frac{S_{l_1}^{\shuffle i_1}\shuffle\ldots\shuffle S_{l_k}^{\shuffle i_k}}{i_1!\ldots i_k!},
&{\displaystyle\mbox{for }w=l_1^{i_1}\ldots l_k^{i_k},\mbox{ with }l_1,\ldots,\atop\displaystyle l_k\in\Lyn\calX,l_1>\ldots>l_k.}
\end{array}\right.
\end{eqnarray}
The graded dual of $\calH_{\shuffle}(\calX)$ is $\calH_{\shuffle}^{\vee}(\calX)
=(\ncp{A}{\calX},{\shuffle},1_{\calX^*},\Delta_{\conc},\epsilon)$.

As an algebra, the module $\ncp{A}{Y}$ is also equipped with 
the associative commutative and unital quasi-shuffle product defined,
for $u,v,w\in Y^*$ and $y_i,y_j\in Y$, by
\begin{eqnarray}
&w\stuffle 1_{Y^*}=1_{Y^*}\stuffle w=w,\\
&y_iu\stuffle y_jv=y_i(u\stuffle y_jv)+y_j(y_iu\stuffle v)+y_{i+j}(u\stuffle v).
\end{eqnarray}
This product also can be dualized according to ($y_k\in Y$)
\begin{eqnarray}\label{Dstuffle}
\Delta_{\stuffle}y_k:=y_k\otimes 1_{Y^*}+1_{Y^*}\otimes y_k +\sum_{i+j=k}y_i\otimes y_j
\end{eqnarray}
which is also a $\conc$-morphism (see \cite{siblings}). We then get another
(connected, graded and co-commutative) bialgebra which, in case $A$ is a $\Q$-algebra, is
isomorphic to the enveloping algebra of the Lie algebra of its primitive elements,
\begin{eqnarray}
\calH_{\stuffle}(Y)=(\ncp{A}{Y},\conc,1_{Y^*},\Delta_{\stuffle},{\tt e})\cong\mathcal{U}(\mathrm{Prim}(\calH_{\stuffle}(Y))),
\end{eqnarray}
where $\mathrm{Prim}(\calH_{\stuffle}(Y))=\mathrm{Im}(\pi_1)=\span_{A}\{\pi_1(w)\vert{w\in Y^*}\}$
and $\pi_1$ is the eulerian projector defined, for any $w\in Y^*$, by \cite{acta,VJM}
\begin{eqnarray}\label{piii_1}
\pi_1(w)=w+\sum_{k=2}^{(w)}\frac{(-1)^{k-1}}k\sum_{u_1,\ldots,u_k\in Y^+}\scal{w}{u_1\stuffle\ldots\stuffle u_k}u_1\ldots u_k,
\end{eqnarray}
and, for any $w=y_{i_i}\ldots y_{i_k}\in Y^*$, $(w)$ denotes the number $i_i+\ldots+i_k$.

\begin{remark}\label{prim}
By \eqref{Dconc} and \eqref{Dshuffle}, any letter $x\in\calX$ is primitive, for $\Delta_{\conc}$ and $\Delta_{\shuffle}$. 
By \eqref{Dstuffle}, the polynomials $\{\pi_1(y_k)\}_{k\ge2}$ and only the letter $y_1$ are primitive, for $\Delta_{\stuffle}$.
\end{remark}

Now, let $\{\Pi_w\}_{w\in Y^*}$ be the linear basis, expanded by decreasing Poincar\'e-Birkhoff-Witt
(PBW for short) after any basis $\{\Pi_l\}_{l\in \Lyn Y}$
of $\mathrm{Prim}(\calH_{\stuffle}(Y))$ homogeneous in weight\footnote{
Factorization \eqref{diagonalY} will be true in particular for the basis \eqref{pii_1}
explicitly constructed there.}, and let $\{\Sigma_w\}_{w\in Y^*}$
be its dual basis which contains the pure transcendence basis $\{\Sigma_l\}_{l\in\Lyn Y}$
of the $A$-quasi-shuffle algebra. One also has the factorization of
the diagonal series ${\cal D}_Y$, on $\calH_{\stuffle}(Y)$,
which reads\footnote{Again all tensor products will be taken over $A$.
Note that this factorization holds for any enveloping algebra as announced
in \cite{reutenauer}. Of course, the diagonal series no 
longer exists and must be replaced by the identity $Id_{\mathcal{U}}$
(see \cite{Duchamp-MO1}, coda for details).} \cite{acta,VJM,CM} 
\begin{eqnarray}\label{diagonalY}
{\cal D}_Y:=\sum_{w\in Y^*}w\otimes w=
\sum_{w\in Y^*}\Sigma_w\otimes\Pi_w=\prod_{l\in\Lyn Y}^{\searrow}e^{\Sigma_l\otimes\Pi_l}.
\end{eqnarray}

We are now in the position to state the following 
\begin{theorem}[\cite{VJM,CM}]\label{isomorphy}
Let $A$ be a $\Q$-algebra, then the endomorphism of algebras
$\varphi_{\pi_1}:(\ncp{A}{Y},\conc,1_{Y^*})\longrightarrow(\ncp{A}{Y},\conc,1_{Y^*})$
mapping $y_k$ to $\pi_1(y_k)$, is an automorphism of $\ncp{A}{Y}$ realizing an isomorphism
of bialgebras between $\calH_{\shuffle}(Y)$ and
\begin{eqnarray*}
\calH_{\stuffle}(Y)\cong\mathcal{U}(\mathrm{Prim}(\calH_{\stuffle}(Y))).
\end{eqnarray*}
In particular, it can be easily checked that the following diagram commutes
\begin{center}
\begin{tikzcd}[column sep=3em]
\ncp{A}{Y}\ar[hook]{r}{\Delta_{\shuffle}}\ar[swap]{d}{\varphi_{\pi_1}}& 
\ncp{A}{Y}\otimes\ncp{A}{Y}\ar[]{d}{\varphi_{\pi_1}\otimes\varphi_{\pi_1}}\\
\ncp{A}{Y}
\ar[hook]{r}{\Delta_{\stuffle}}& 
\ncp{A}{Y}\otimes\ncp{A}{Y} 
\end{tikzcd}
\end{center}
\end{theorem}

Hence, the bases $\{\Pi_w\}_{w\in Y^*}$ and $\{\Sigma_w\}_{w\in Y^*}$
of $\mathcal{U}(\mathrm{Prim}(\calH_{\stuffle}(Y)))$ are images by
$\varphi_{\pi_1}$ and by the adjoint mapping of its inverse, $\check\varphi_{\pi_1}^{-1}$
of $\{P_w\}_{w\in Y^*}$ and $\{S_w\}_{w\in Y^*}$, respectively.
Algorithmically, by Remark \ref{prim}, the dual bases of homogenous polynomials $\{\Pi_w\}_{w\in Y^*}$
and $\{\Sigma_w\}_{w\in Y^*}$ can be constructed directly and recursively by
\begin{eqnarray}\label{pii_1}
\left\{\begin{array}{llll}
\Pi_{y_s}=\pi_1(y_s),
&\Sigma_{y_s}=y_s
&\mbox{for }y_s\in Y,\\
\Pi_{l}=[\Pi_{l_1},\Pi_{l_2}],
&\Sigma_l=\Sum_{(*)}\frac{y_{s_{k_1}+\ldots+s_{k_i}}}{i!}\Sigma_{l_1\ldots l_n},
&{\displaystyle\mbox{for }l\in\Lyn Y-Y\atop\displaystyle st(l)=(l_1,l_2),}\\
\Pi_{w}=\Pi_{l_1}^{i_1}\ldots\Pi_{l_k}^{i_k},
&\Sigma_w=\displaystyle\frac{\Sigma_{l_1}^{\stuffle i_1}\stuffle\ldots\stuffle\Sigma_{l_k}^{\stuffle i_k}}{i_1!\ldots i_k!},
&{\displaystyle\mbox{for }w=l_1^{i_1}\ldots l_k^{i_k},\mbox{ with }l_1,\ldots,\atop\displaystyle l_k\in\Lyn Y,l_1>\ldots>l_k.}
\end{array}\right.
\end{eqnarray}
In $(*)$, the sum is taken over all $\{k_1,\ldots,k_i\}\subset\{1,\ldots,k\}$
and $l_1\ge\ldots\ge l_n$ such that $(y_{s_1},\ldots,y_{s_k})\stackrel{*}
{\Leftarrow}(y_{s_{k_1}},\ldots,y_{s_{k_i}},l_1,\ldots,l_n)$, where
$\stackrel{*}{\Leftarrow}$ denotes the transitive closure of the relation
on standard sequences, denoted by $\Leftarrow$ \cite{SLC74,reutenauer}.

To end this section, let us extend $\conc$ and $\shuffle$, for any series $S,R\in\ncs{A}{\calX}$, by
\begin{eqnarray}\label{extenoverseries}
SR=\sum_{w\in\calX^*}\biggl(\sum_{u,v\in\calX^*\atop uv=w}\scal{S}{u}\scal{R}{v}\biggr)w
&\mbox{and}&S\shuffle R=\sum_{u,v\in\calX^*}\scal{S}{u}\scal{R}{v}u\shuffle v,
\end{eqnarray}
and $\stuffle$, for any series $S,R\in\ncs{A}{Y}$, by
\begin{eqnarray}
S\stuffle R=\sum_{u,v\in Y^*}\scal{S}{u}\scal{R}{v}u\stuffle v.
\end{eqnarray}
Let us also extend the coproduct $\Delta_{\stuffle}$ (resp. $\Delta_{\conc}$ and $\Delta_{\shuffle}$)
given in \eqref{Dstuffle} (resp. \eqref{Dconc} and \eqref{Dshuffle})
over $\ncs{A}{Y}$ (resp. $\ncs{A}{\calX}$) by linearity as follows
\begin{eqnarray}
\forall S\in\ncs{A}{Y},\quad
\Delta_{\stuffle}S&=&\sum_{w\in Y^*}\scal{S}{w}\Delta_{\stuffle}w\in{\ncs{A}{Y^*\otimes Y^*}},\label{D1}\\
\forall S\in\ncs{A}{\calX},\quad
\Delta_{\shuffle}S&=&\sum_{w\in\calX^*}\scal{S}{w}\Delta_{\shuffle}w\in{\ncs{A}{\calX^*\otimes\calX^*}},\label{D2}\\
\forall S\in\ncs{A}{\calX},\quad
\Delta_{\conc}S&=&\sum_{w\in\calX^*}\scal{S}{w}\Delta_{\conc}w\in{\ncs{A}{\calX^*\otimes\calX^*}}.\label{D3}
\end{eqnarray}

\begin{remark}
In \eqref{D1}--\eqref{D3}, if $A=K$, being a field, then $\ncs{K}{\calX}\otimes\ncs{K}{\calX}$
embeds (injectively) in $\ncs{K}{\calX^*\otimes\calX^*}\cong\ncs{[\ncs{K}{\calX}]}{\calX}$.
Indeed, $\ncs{K}{\calX}\otimes\ncs{K}{\calX}$ contains the elements of the form
$\sum_{i\in I}G_i\otimes D_i$, for some $I$ finite and $(G_i,D_i)\in\ncs{K}{\calX}\times\ncs{K}{\calX}$.
But $\sum_{i\ge0}u^i\otimes v^i$ belongs to $\ncs{K}{\calX^*\otimes\calX^*}$
and does not belong to $\ncs{K}{\calX}\otimes\ncs{K}{\calX}$, for $u,v\in\calX^{\ge1}$.
\end{remark}

\begin{definition}\label{dec0}
Any series $S\in\ncs{A}{Y}$ (resp. $\ncs{A}{\calX}$) is said to be
\begin{enumerate}
\item a $\stuffle$ (resp. $\shuffle,\conc$)-character of $(\ncp{A}{X},\conc,1_{X^*})$ iff,
for any $u,v\in Y^*$ (resp. $\calX^*$), one has $\scal{S}{1_{Y^*}}=1_A$
(resp. $\scal{S}{1_{\calX^*}}=1_A$) and $\scal{S}{u\stuffle v}=\scal{S}{u}\scal{S}{v}$
(resp. $\scal{S}{u\shuffle v}=\scal{S}{u}\scal{S}{v},\scal{S}{uv}=\scal{S}{u}\scal{S}{v}$).
\item an infinitesimal $\stuffle$ (resp. $\shuffle,\conc$)-character of $(\ncp{A}{X},\conc,1_{X^*})$
iff, for any $u,v\in Y^*$ (resp. $\calX^*$), one has
$\scal{S}{u\stuffle v}=\scal{S}{u}\scal{v}{1_{Y^*}}+\scal{u}{1_{Y^*}}\scal{S}{v}$
(resp. $\scal{S}{u\shuffle v}=\scal{S}{u}\scal{v}{1_{Y^*}}+\scal{u}{1_{Y^*}}\scal{S}{v},
\scal{S}{uv}=\scal{S}{u}\scal{v}{1_{Y^*}}\allowbreak+\scal{u}{1_{Y^*}}\scal{S}{v}$).
Moreover, if $\scal{S}{1_{Y^*}}=1_A$ (resp. $\scal{S}{1_{\calX^*}}=1_A$)
then $\scal{S}{u\stuffle v}=0$ (resp. $\scal{S}{u\shuffle v}=0,\scal{S}{uv}=0$).
\end{enumerate}
\end{definition}

Now, let us see how these combinatorics will operate over rational series to describe,
as illustration, solutions of linear differential equations in Section \ref{PVNC} below.

\subsection{Representative series}
Representative (or rational) series are the representative functions on the free monoid\footnote{
These functions were considered on groups in \cite{Cartier,ChariPressley}.} \cite{SDSC}
and their magic is that it rests on four (apparently distant) pillars:
\begin{itemize}
\item Separated coproduct (SC) uniquely for fields,
\item Finite orbit by shifts (FS), 
\item Result of a rational expression (RE), 
\item Linear representation (LR).
\end{itemize}

\begin{definition}\label{dec1}
Let $S\in\ncs{A}{\calX}$ (resp. $\ncp{A}{\calX}$) and $P\in\ncp{A}{\calX}$ (resp. $\ncs{A}{\calX}$).
\begin{enumerate}
\item The {\it left} (resp. {\it right}) \textit{shift}\footnote{Some schools (as Jacob one, see \cite{jacob,hespel})
used to call this a {\it residual}. These actions are none other than the shifts of functions of harmonic analysis.}
of $S$ by $P$, is $P\rd S$ (resp. {$S\rg P$}) defined by\footnote{They are associative, commute with each other:
$S\rg(PR)=(S\rg P)\rg R,P\rd(R\rd S)=(P.R)\rd S$ and $(P\rg S)\rd R=P\rg(S\rd R)$ and, for $x,y\in\calX,w\in\calX^*$,
$x\rd(wy)=(yw)\rg x=\delta_x^yw$ (Kronecker delta).} 
\begin{eqnarray*}
\forall w\in\calX^*,&\scal{{P\rd S}}{w}=\scal{S}{wP}&(\mbox{resp. }\scal{{S\rg P}}{w}=\scal{S}{Pw}).
\end{eqnarray*}

\item For any $S\in\ncs{A}\calX$ such that $\scal{S}{1_{\calX^*}}=0$, the Kleene star of $S$ is defined as\footnote{
Using one of the topologies of section \ref{preliminaries} (adapted with $A$ replacing $\calH(\Omega)$), we have
$S^*=\sum_{n\ge0}S^n$. We also get the fact that the space $\widehat{A.\calX}$ (used below) of series of degree $1$,
\textit{i.e.} the set $\{\sum_{x\in\cal X}\alpha(x)x\}_{\alpha\in A^X}$ is the closure of the $A$-module $A.\calX$
generated by letters. In the case of a finite alphabet however (here $\calX=X$) \cite{SDSC}, $\widehat{A.\calX}=A.\calX$.}
$S^*=(1-S)^{-1}$.

\item\label{iii} In case $A=K$ is a field, one can define also the Sweedler's dual
$\calH_{\shuffle}^{\circ}(\calX)$ of $\calH_{\shuffle}(\calX)$ by
$S\in\calH_{\shuffle}^{\circ}(\calX)\iff\Delta_{\conc}(S)=\sum_{i\in I}G_i\otimes D_i$
\cite{reutenauer}, for some $I$ finite, $\{G_i\}_{i\in I};\{D_i\}_{i\in I}$ being series
(as a matter of fact, it can be shown that they even can been choosen in
$\calH_{\shuffle}^{\circ}(\calX)$, see \cite{CM}).
\end{enumerate}
\end{definition}

\begin{theorem}[\cite{DR,SDSC,orlando,reutenauer}]\label{residual}
For $S\in\ncs{A}{\calX}$, the following assertions are equivalent\footnote{
When $A$ is noetherian, first condition is equivalent to the fact that the module
generated by $\{S\trl w\}_{w\in \calX^*}$ (resp. $\{w\rd S\}_{w\in \calX^*}$)
is finitely generated (and more precisely, in this case, by a finite
number of those shifts). Unfortunately we are not in this case here,
but our ring being without zero divisors (analytic functions),
we can use the fraction field, here being realized by germs \cite{Linz}.}
\begin{enumerate}
\item\label{i} The shifts $\{S\trl w\}_{w\in \calX^*}$ (resp. $\{w\rd S\}_{w\in \calX^*}$)
lie in a finitely generated shift-invariant $A$-module\footnote{see \cite{jacob2}.}.

\item The series $S$ belongs to the (algebraic) closure of $\widehat{A.\calX}$
by the operations $\{\conc,+,*\}$ (within $\ncs{A}\calX$). 

\item There is a linear representation $(\nu,\mu,\eta)$, of rank $n$, for $S$ with
$\nu\in M_{1,n}(A),\allowbreak\eta\in M_{n,1}(A)$ and a morphism of monoids
$\mu:\calX^*\rightarrow M_{n,n}(A)$ such that
\begin{eqnarray*}
S=\sum_{w\in\calX^*}\bigl(\nu\mu(w)\eta\bigr)w.
\end{eqnarray*}
\end{enumerate}
\end{theorem}
A series satisfying one of the conditions of Theorem \ref{residual} is called \textit{rational}.
The set of these series, a $A$-module\footnote{In fact (we will see it) a unital $A$-algebra for
$\conc$ and $\shuffle.$} denoted by ${\ncs{A^{\mathrm{rat}}}\calX}$, isclosed by $\{\conc,+,*\}$.

We also have the following constructions of linear representations (only the last one is new and
the first ones are already treated in \cite{DFLL,jacob}). In particular, the constructions of
$R_1\shuffle R_2$ and $R_1\stuffle R_2$ base on the tensor products (of representations) and use
the expressions of the coproducts given in \eqref{Dshuffle} and \eqref{Dstuffle}, respectively.

\begin{proposition}\label{linearrepresentation}
The module ${\ncs{A^{\mathrm{rat}}}\calX}$ (resp. ${\ncs{A^{\mathrm{rat}}}{Y}}$)
is closed by $\shuffle$ (resp. $\stuffle$). Moreover, for $i=1,2$, let
$R_i\in{\ncs{A^{\mathrm{rat}}}\calX}$ and $(\nu_i,\mu_i,\eta_i)$ be its
representation of dimension $n_i$. Then the linear representation of
\begin{eqnarray*}
R_i^*&\mbox{is}&\biggl(\begin{pmatrix}0&1\end{pmatrix},
\left\{\begin{pmatrix}\mu_i(x)+\eta_i\nu_i\mu_i(x)&0\cr\nu_i\eta_i&0\end{pmatrix}\right\}_{x\in\calX},
\begin{pmatrix}\eta_i\cr1\end{pmatrix}\biggr),\cr
\mbox{that of }R_1+R_2&\mbox{is}&\biggl(\begin{pmatrix}\nu_1&\nu_2\end{pmatrix},
\left\{\begin{pmatrix}\mu_1(x)&{\bf 0}\cr{\bf 0}&\mu_2(x)\end{pmatrix}\right\}_{x\in\calX},
\begin{pmatrix}\eta_1\cr\eta_2\end{pmatrix}\biggr),\cr
\mbox{that of }R_1R_2&\mbox{is}&\biggl(\begin{pmatrix}\nu_1&0\end{pmatrix},
\left\{\begin{pmatrix}\mu_1(x)&\eta_1\nu_2\mu_2(x)\cr0&\mu_2(x)\end{pmatrix}\right\}_{x\in\calX},
\begin{pmatrix}\eta_1\mu_2\eta_2\cr\eta_2\end{pmatrix}\biggr),\cr
\mbox{that of }R_1\shuffle R_2&\mbox{is}&(\nu_1\otimes\nu_2,\{\mu_1(x)\otimes\mathrm{I}_{n_2}
+\mathrm{I}_{n_1}\otimes\mu_2(x)\}_{x\in\calX},\eta_1\otimes\eta_2),\\
\mbox{that of }R_1\stuffle R_2&\mbox{is}&(\nu_1\otimes\nu_2,
\{\mu_1(y_k)\otimes\mathrm{I}_{n_2}+\mathrm{I}_{n_1}\otimes\mu_2(y_k)\cr
&&+\sum_{i+j=k}\mu_1(y_i)\otimes\mu_2(y_j)\}_{k\ge1},\eta_1\otimes\eta_2).
\end{eqnarray*}
\end{proposition}

\begin{example}[Identity $(-t^2x_0x_1)^*\shuffle(t^2x_0x_1)^*=(-4t^4x_0^2x_1^2)^*$, \cite{words03,orlando}]
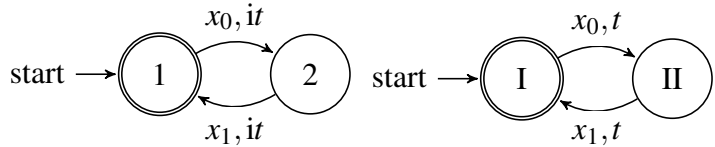
\label{quiplait}
\begin{eqnarray*}
\begin{tikzpicture}[->,>=stealth',shorten >=1pt,auto,node distance=2cm,
                    semithick,every node/.style={fill=white}]				
  \node[initial,state,accepting] (A)                    {$1$};  
	\node[state]                   (B) [right of=A] {$2$};
  \path (A) edge [bend left] node {$x_0,\mathrm{i}t$} (B)
        (B) edge [bend left] node {$x_1,\mathrm{i}t$}  (A);
\end{tikzpicture}
&&
\begin{tikzpicture}[->,>=stealth',shorten >=1pt,auto,node distance=2cm,
                    semithick,every node/.style={fill=white}]				
  \node[initial,state,accepting] (1)                    {$\mathrm{I}$};  
	\node[state]                   (2) [right of=A] {$\mathrm{II}$};
  \path (A) edge [bend left] node {$x_0,t$} (B)
        (B) edge [bend left] node {$x_1,t$}  (A);
\end{tikzpicture}\cr
(-t^2x_0x_1)^*\leftrightarrow(\nu_2,\{\mu_2(x_0),\mu_2(x_1)\},\eta_2)
&&
(t^2x_0x_1)^*\leftrightarrow(\nu_1,\{\mu_1(x_0),\mu_1(x_1)\},\eta_1).
\end{eqnarray*}
\begin{eqnarray*}
\begin{array}{rlrl}
\nu_1=\begin{pmatrix}1&0\end{pmatrix},
&\mu_1(x_0)=\begin{pmatrix}0&t\cr0&0\end{pmatrix},
&\mu_1(x_1)=\begin{pmatrix}0&0\cr t&0\end{pmatrix},
&\eta_1=\begin{pmatrix}1\cr0\end{pmatrix},\cr
\nu_2=\begin{pmatrix}1&0\end{pmatrix},
&\mu_2(x_0)=\begin{pmatrix}0&\mathrm{i}t\cr0&0\end{pmatrix},
&\mu_2(x_1)=\begin{pmatrix}0&0\cr\mathrm{i}t&0\end{pmatrix},
&\eta_2=\begin{pmatrix}1\cr0\end{pmatrix}.
\end{array}
\end{eqnarray*}

\begin{eqnarray*}
(-t^2x_0x_1)^*\shuffle(t^2x_0x_1)^*&\leftrightarrow&(\nu,\{\mu(x_0),\mu(x_1)\},\eta)\\
&=&(\nu_1\otimes\nu_2,\{\mu_1(x_0)\otimes\mathrm{I}_{n_2}+\mathrm{I}_{n_1}\otimes\mu_2(x_0),\\
&&\mu_1(x_1)\otimes\mathrm{I}_{n_2}+\mathrm{I}_{n_1}\otimes\mu_2(x_1),\eta_1\otimes\eta_2).
\end{eqnarray*}

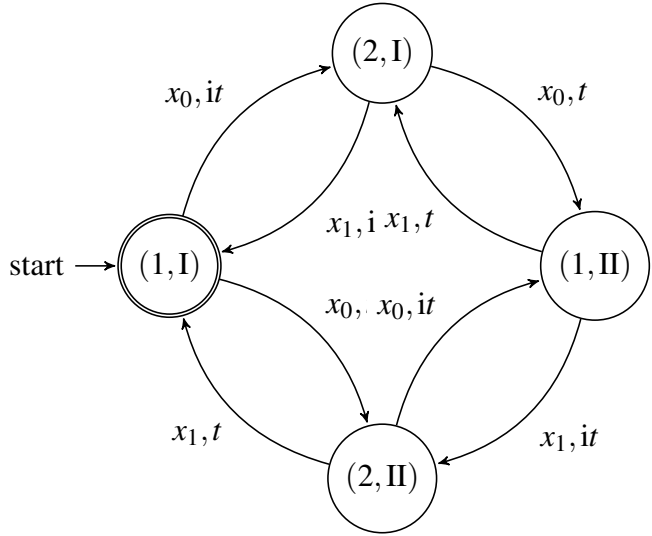
\begin{eqnarray*}
\begin{tikzpicture}[->,>=stealth',shorten >=1pt,auto,node distance=4cm,
                    semithick,every node/.style={fill=white}]				
	\node[initial,state,accepting] (A)                    {$(1,\mathrm{I})$};
  \node[state]                   (B) [above right of=A] {$(2,\mathrm{I})$};
  \node[state]                   (D) [below right of=A] {$(2,\mathrm{II})$};
  \node[state]                   (C) [below right of=B] {$(1,\mathrm{II})$};
  \path (A) edge [bend left] node {$x_0,\mathrm{i}t$} (B)
            edge [bend left] node {$x_0,t$}           (D)
        (B) edge [bend left] node {$x_0,t$}           (C)
				    edge [bend left] node {$x_1,\mathrm{i}t$} (A)
        (C) edge [bend left] node {$x_1,t$}           (B)
				    edge [bend left] node {$x_1,\mathrm{i}t$} (D)
        (D) edge [bend left] node {$x_1,t$}           (A)
            edge [bend left] node {$x_0,\mathrm{i}t$} (C);
\end{tikzpicture}
\end{eqnarray*}

\begin{eqnarray*}
\nu=\begin{pmatrix}1&0&0&0\end{pmatrix},
&\begin{array}{c}
\mu(x_0)=\begin{pmatrix}
0&0&t&0\cr
0&0&0&t\cr
0&0&0&0\cr
0&0&0&0
\end{pmatrix}
+
\begin{pmatrix}
0&\mathrm{i}t&0&0\cr
0&0&0&0\cr
0&0&0&\mathrm{i}t\cr
0&0&0&0
\end{pmatrix}
=\begin{pmatrix}
0&\mathrm{i}t&t&0\cr
0&0&0&t\cr
0&0&0&\mathrm{i}t\cr
0&0&0&0
\end{pmatrix},\\
\mu(x_1)=\begin{pmatrix}
0&0&0&0\cr
0&0&0&0\cr
t&0&0&0\cr
0&t&0&0
\end{pmatrix}
+
\begin{pmatrix}
0&0&0&0\cr
\mathrm{i}t&0&0&0\cr
0&0&0&0\cr
0&0&\mathrm{i}t&0
\end{pmatrix}
=\begin{pmatrix}
0&0&0&0\cr
\mathrm{i}t&0&0&0\cr
t&0&0&0\cr
0&t&\mathrm{i}t&0
\end{pmatrix},
\end{array}
&\eta=\begin{pmatrix}1\cr0\cr0\cr0\end{pmatrix}.
\end{eqnarray*}
\end{example}

\begin{proposition}
With the notations of Definition \ref{dec1}.\eqref{iii}, if $A$ is a field $K$ then 
\begin{enumerate}
\item[{\rm (a)}] Assertions of Theorem \ref{residual} are equivalent to
\begin{enumerate}
\item[\rm{(iv)}] There exists a finite double family of series
$(G_i,D_i)_{i\in F}$ such that\footnote{See \cite{CM} for a way to obtain
this finite double family of series $(G_i,D_i)_{i\in F}$.}
\begin{eqnarray*}
\Delta_{\conc}S=\sum_{i\in F}G_i\otimes D_i
\end{eqnarray*}
\end{enumerate}
\item[{\rm (b)}] For $S\in\calH_{\shuffle}^{\circ}(\calX)$, since $A$ is a field
then the previous identity is equivalent to
\begin{eqnarray*}
\forall P,Q\in\calH_{\shuffle}(\calX),&&
\scal{S}{PQ}=\sum_{i\in I}\scal{G_i}{P}\scal{D_i}{Q}.
\end{eqnarray*}
\end{enumerate}
Therefore, $(\ncs{{K}^{\mathrm{rat}}}{\calX},\shuffle,1_{\calX^*},\Delta_{\conc},{\tt e})$
(resp. $(\ncs{{K}^{\mathrm{rat}}}{Y},\allowbreak\stuffle,\allowbreak1_{\calX^*},\Delta_{\conc},{\tt e})$)
is the Sweedler's dual of $\calH_{\shuffle}(\calX)$ (resp. $\calH_{\stuffle}(Y)$).
\end{proposition}

\section{Triangularity, solvability and rationality}\label{TST}
\subsection{Syntactically exchangeable rational series}\label{TST1}

Now, we have to study a special set of series in order to work with the rational series of this class:
a series $S\in\ncs{A}{\calX}$ is called {\it syntactically exchangeable}
if and only if it is constant on multi-homogeneous classes, \textit{i.e.}
\begin{eqnarray}
(\forall u,v\in\calX^*)([(\forall x\in\calX)(|u|_x=|v|_x)]&\Rightarrow&\scal{S}{u}=\scal{S}{v}).
\end{eqnarray}
A series $S\in\ncs{A}{{\cal X}}$ is syntactically exchangeable iff it is of the following form
\begin{eqnarray}\label{cylindric}
S=\sum_{\alpha\in\N^{(\calX)},\mathrm{supp}(\alpha)=\{x_1,\ldots,x_k\}}s_{\alpha},
x_1^{\alpha(x_1)}\shuffle \ldots\shuffle  x_k^{\alpha(x_k)}.
\end{eqnarray}
The set of these series, a shuffle subalgebra of $\ncs{A}{X}$, will be denoted $\ncs{A_{\mathrm{exc}}^{\mathrm{synt}}}{\calX}$ .

When $A$ is a field, the rational and exchangeable series are exactly those who admit a representation
with commuting matrices (at least the minimal one is such, see Theorem \ref{exchangeable} below).
We will take this as a definition as, even for rings, this property implies syntactic exchangeability. 

\begin{definition}
A series $S\in \ncs{A^{\mathrm{rat}}}{\calX}$ will be said \textit{rationally exchangeable} if it admits a representation
$(\nu,\mu,\eta)$ such that $\{\mu(x)\}_{x\in \calX}$ is a set of commuting matrices, the set of these series, a shuffle
subalgebra of $\ncs{A}{X}$, will be denoted $\ncs{A_{\mathrm{exc}}^{\mathrm{rat}}}{\calX}$.
\end{definition}   

\begin{theorem}[See \cite{KSO,CM}]\label{exchangeable}
Let $\ncs{A_{\mathrm{exc}}^{\mathrm{synt}}}{\calX}$ denote the set of (syntactically) exchangeable series. Then
\begin{enumerate}
\item\label{ratex1} In all cases, one has $\ncs{A_{\mathrm{exc}}^{\mathrm{rat}}}{\calX}
\subset\ncs{A^{\mathrm{rat}}}{\calX}\cap\ncs{A_{\mathrm{exc}}^{\mathrm{synt}}}{\calX}$.
The equality holds when $A$ is a field and
\begin{eqnarray*}
\ncs{A_{\mathrm{exc}}^{\mathrm{rat}}}{X}
=&\ncs{A^{\mathrm{rat}}}{x_0}\shuffle\ncs{A^{\mathrm{rat}}}{x_1}&=\shuffle_{x\in X}\ncs{A^{\mathrm{rat}}}{x},\\
\ncs{A_{\mathrm{exc}}^{\mathrm{rat}}}{Y}\cap\ncs{A_{\mathrm{fin}}^{\mathrm{rat}}}{Y}
=&\bigcup\limits_{k\ge0}\ncs{A^{\mathrm{rat}}}{y_1}\shuffle\ldots\shuffle\ncs{A^{\mathrm{rat}}}{y_k}
&\subsetneq\ncs{A_{\mathrm{exc}}^{\mathrm{rat}}}{Y},
\end{eqnarray*}
where $\ncs{A_{\mathrm{fin}}^{\mathrm{rat}}}{Y}=\cup_{F\subset_{finite}Y}\ncs{A^{\mathrm{rat}}}{F}$,
the algebra of series over finite subalphabets\footnote{\label{starofplane}The last inclusion is strict as shows
the example of the following identity \cite{ladji}
\begin{eqnarray*}
(ty_1+t^2y_2+\ldots)^*=\lim\limits_{k\rightarrow+\infty}(ty_1+\ldots+t^ky_k)^*
=\lim\limits_{k\rightarrow+\infty}(ty_1)^*\shuffle\ldots\shuffle(t^ky_k)^*=\shuffle_{k\ge1}(t^ky_k)^*
\end{eqnarray*}
which lives in $\ncs{A_{\mathrm{exc}}^{\mathrm{rat}}}{Y}$ but not in 
$\ncs{A_{\mathrm{exc}}^{\mathrm{rat}}}{Y}\cap \ncs{A_{\mathrm{fin}}^{\mathrm{rat}}}{Y}$.}.
\item\label{Kronecker} (Kronecker's theorem \cite{berstel,zygmund})
One has $\ncs{A^{\mathrm{rat}}}{x}=\{P(1-xQ)^{-1}\}_{P,Q\in A[x]}$ (for $x\in\calX$)
and if $A=K$ is an algebraically closed field of characteristic zero one also has
$\ncs{K^{\mathrm{rat}}}{x}=\span_K\{(ax)^*\shuffle\ncp{K}{x}\vert a\in K\}$.
\item\label{conccharacter} The rational series $(\sum_{x\in \calX}\alpha_x\,x)^*$ are $\conc$-characters and any
$\conc$-character is of this form.
\item\label{indepchargen} Let us suppose that $A$ is without zero divisors and let $(\varphi_i)_{i\in I}$
be a family within $\widehat{A\calX}$ which is $\Z$-linearly independent then, the family 
$\Lyn(\calX)\uplus\{\varphi_i^*\}_{i\in I}$ is algebraically free over $A$ within
$(\ncs{A^{\mathrm{rat}}}{\calX},\shuffle,1_{\calX^*})$. 
\item\label{indepchar} In particular, if $A$ is a ring without zero divisors
$\{x^*\}_{x\in\calX}$ (resp. $\{y^*\}_{y\in Y}$) are algebraically independent over
$(\ncp{A}{\calX},\shuffle,1_{\calX^*})$ (resp. $(\ncp{A}{Y},\stuffle,1_{Y^*})$)
within $(\ncs{A^{\mathrm{rat}}}{\calX},\shuffle,1_{\calX^*})$
(resp. $(\ncs{A^{\mathrm{rat}}}{Y},\stuffle,1_{Y^*})$).
\end{enumerate}
\end{theorem}

\begin{proof}
\begin{enumerate}
\item The inclusion is obvious in view of \eqref{cylindric}.
For the equality, it suffices to prove that, when $A$ is a field,  every rational and
exchangeable series admits a representation with commuting matrices. This is true of
any minimal representation as shows the computation of shifts (see \cite{DR,KSO,CM}).

Now, if $\calX$ is finite, as all matrices commute, we have 
\begin{eqnarray*}
\sum_{w\in \calX^*}\mu(w)w=\biggl(\sum_{x\in\calX}\mu(x)x\biggr)^*=\shuffle_{x\in \calX}(\mu(x)x)^*
\end{eqnarray*}
and the result comes from the fact that $R$ is a linear combination of matrix elements.
As regards the second equality, inclusion $\supset$ is straightforward. We remark that  the union 
$\bigcup_{k\ge1}\ncs{A^{\mathrm{rat}}}{y_1}\shuffle\ldots\shuffle\ncs{A^{\mathrm{rat}}}{y_k}$
is directed as these algebras are nested in one another. With this in view, the reverse inclusion comes from the fact that every
$S\in\ncs{A_{\mathrm{fin}}^{\mathrm{rat}}}{Y}$ is a series over a finite alphabet and the result follows from the first equality.
\item Let $\calA=\{P(1-xQ)^{-1}\}_{P,Q\in A[x]}$. Since $P(1-xQ)^{-1}=P(xQ)^{*}$ then it is obvious that
$\calA\subset \ncs{A^{\mathrm{rat}}}{x}$. Next, it is easy to check that $\calA$ contains $\ncp{A}{x}(=A[x])$
and it is closed by $+,\conc$ as, for instance,
\begin{eqnarray*}
(1-xQ_1)(1-xQ_2)=(1-x(Q_1+Q_2-xQ_1Q_2)).
\end{eqnarray*}
We also have to prove that $\calA$ is closed for $*$. For this to be applied to $P(1-xQ)^{-1}$,
we must suppose that $P(0)=0$ (as, indeed, $\scal{P(1-xQ)^{-1}}{1_{x^*}}=P(0)$) and,
in this case, $P=xP_1$. Now
\begin{eqnarray*}
\Big(\frac{P}{1-xQ}\Big)^*=\Big(1-\frac{P}{1-xQ}\Big)^{-1}=\frac{1-xQ}{1-x(Q+P_1)}\in\calA.
\end{eqnarray*}

\item Let $S=(\sum_{x\in \calX}\alpha_x\,x)^*$ and since $S=1+(\sum_{x\in \calX}\alpha_xx)S$
then $\scal{S}{1_{\calX^*}}=1_A$. If $w=xu$ then we get $\scal{S}{xu}=\alpha_x\scal{S}{u}$.
Thus, by recurrence  on the length, $\scal{S}{x_1\ldots x_k}=\prod_{i=1}^k\alpha_{x_i}$
showing that $S$ is a $\conc$-character. Conversely, by Sch\"utzenberger's reconstruction
lemma saying that, for any series $S$
\begin{eqnarray*}
S=\scal{S}{1_{\calX^*}}.1_A+\sum_{x\in\calX}x.x^{-1}S
\end{eqnarray*}
but, if $S$ is a $\conc$-character, $\scal{S}{1_{\calX^*}}=1$ and $x^{-1}S=\scal{S}{x}S$,
then the previous expression reads $S=1_A+(\sum_{x\in\calX}\scal{S}{x}x)S$. The last 
equality is equivalent to $S=(\sum_{x\in\calX}\scal{S}{x}.x)^*$ proving the claim. 
\item As $(\ncp{A}{\calX},\shuffle,1_{\calX^*})$ and $(\ncp{A}{Y},\stuffle,1_{Y^*})$
are enveloping algebras, this property is an application of the fact that, on an enveloping 
$\mathcal{U}$, the characters are linearly independent w.r.t. to the convolution algebra 
$\mathcal{U}^*_{\infty}$ (see the general construction and proof in \cite{DGM} or \cite{ICP}).
Here, this  convolution algebra ($\mathcal{U}^*_{\infty}$) contains the polynomials
(is equal in case of finite $\calX$). Now, consider a monomial 
\begin{eqnarray*}
(\varphi_{i_1}^*)^{\shuffle\alpha_1}\ldots(\varphi_{i_n}^*)^{\shuffle\alpha_n}=
\biggl(\sum_{k=1}^n\alpha_{i_k}\varphi_{i_k}\biggr)^*
\end{eqnarray*}   
The $\Z$-linear independence of the monomials in $(\varphi_i)_{i\in I}$
implies that all these monomials are linearly independent over $\ncp{A}{\calX}$
which proves  algebraic independence of the family $(\varphi_i)_{i\in I}$.

To end with, the fact that $\Lyn(\calX)\uplus\{\varphi_i^*\}_{i\in I}$ is algebraically free comes 
from Radford theorem $(\ncp{A}{\calX},\shuffle,1_{\calX^*})\simeq A[\Lyn(\calX)]$ and the transitivity 
of polynomial algebras (see \cite{B_Alg1} ch III.2 Proposition 8).
\item Comes directly as an application of the preceding point.
\end{enumerate}
\end{proof}
\begin{remark}(Point \eqref{Kronecker} of Theorem \ref{exchangeable} above)
Kronecker's theorem which can be rephrased in terms of stars as 
$\ncs{A^{\mathrm{rat}}}{x}=\{P(xQ)^*\}_{P,Q\in A[x]}$
holds for every ring and is therefore characteristic free, unlike the 
shuffle version requiring algebraic closure and denominators.
\end{remark}

Now, we are in situation to characterize $\conc$-characters and infinitesimal $\conc$-characters
(see Definition \ref{dec0}) of $(\ncp{A}{\calX},\conc,1_{\calX^*})$.

\begin{corollary}[Kleene stars of the plane]\label{Kleene}
Let $R,L\in\ncs{A^{\mathrm{rat}}}{\calX},\scal{R}{1_{\calX^*}}=1_A,\allowbreak\scal{L}{1_{\calX^*}}=0$,
such that $L^*=R$. The following assertions are equivalent
\begin{enumerate}
\item\label{caracter} $R$ is a $\conc$-character of $(\ncp{A}{\calX},\conc,1_{\calX^*})$.
\item\label{exp_rat} There is a family of coefficients $(c_x)_{x\in\calX}$ such that $R=(\sum_{x\in\calX}c_xx)^*$.
\item\label{lin_rep} The series $R$ admits a linear representation of dimension one\footnote{
The dimension is here (as in \cite{berstel}) the size of the matrices.}.
\item\label{plane} $L$ belongs to the plane ${A.\calX}$.
\item\label{log} $L$ is an infinitesimal $\conc$-character of $(\ncp{A}{\calX},\conc,1_{X^*})$.
\end{enumerate}
Moreover, $(\alpha_0x_0+\alpha_1x_1)^*\shuffle(\beta_0x_0+\beta_1x_1)^*=
((\alpha_0+\beta_0)x_0+(\alpha_1+\beta_1)x_1)^*$ and\footnote{In particular, 
$(a_sy_s)^*\stuffle(a_ry_r)^*=(a_sy_s+a_ry_r+a_sa_ry_{s+r})^*$
and $(a_sy_s)^*\stuffle(-a_sy_s)^*=(-a_s^2y_{2s})^*$.}
\begin{eqnarray*}
\biggl(\sum_{s\ge1}a_sy_s\biggr)^*\stuffle\biggl(\sum_{s\ge1}b_sy_s\biggr)^*
=\biggl(\sum_{s\ge1}(a_s+b_s)y_s+\sum_{r,s\ge1}a_sb_ry_{s+r}\biggr)^*,
\end{eqnarray*}
where, for any $i=0,1$ and $s\ge1$, $\alpha_i,\beta_i,a_s,b_s\in\C$.
\end{corollary}

\begin{proof}
\underline{\eqref{caracter} $\Leftrightarrow$ \eqref{exp_rat}:}
This corresponds to the point \eqref{conccharacter} of Theorem \ref{exchangeable} above.

\underline{\eqref{exp_rat} $\Leftrightarrow$ \eqref{lin_rep}}
This is a direct consequence of Theorem \ref{residual}.

\underline{\eqref{exp_rat} $\Leftrightarrow$ \eqref{plane}:}
This is obvious, by construction (in which $L$ is viewed as the
$\shuffle$-logarithm of $R$). Indeed, doing as in Note \ref{starofplane}, one has
\begin{eqnarray*}
R=\Big(\sum_{x\in\calX}c_xx\Big)^*=\shuffle_{x\in\calX}(c_xx)^*=
\shuffle_{x\in\calX}\exp_{\shuffle}(c_xx)=\exp_{\shuffle}\Big(\sum_{x\in\calX}c_xx\Big).
\end{eqnarray*}

\underline{\eqref{plane} $\Leftrightarrow$ \eqref{log}:}
If $L$ is an infinitesimal character then, by Definition \ref{dec0}, $\scal{L}{uv}=
\scal{L}{u}\scal{v}{1_{\calX^*}}+\scal{u}{1_{\calX^*}}\scal{L}{v}$, for $u,v\in\calX^*$. 
Hence, for $w=uv\in\calX^{\ge2}$ with $u,v\in\calX^+$, one gets $\scal{L}{w}=\scal{L}{uv}=0$.
In addition, for $u=v=1_{\calX^*}$, one also gets $\scal{L}{1_{\calX^*}}=0$ and it follows that
$L=\sum_{x\in\calX}\scal{L}{x}x$. Conversely, since for any $u,v\in\calX^*$ and $x\in\calX$,
one has $\scal{uv}{x}=\scal{u}{x}\scal{v}{1_{\calX^*}}+\scal{u}{1_{\calX^*}}\scal{v}{x}=0$
then, by the pairing in \eqref{pairing}, one deduces that
\begin{eqnarray*}
\scal{L}{uv}=\sum_{x\in\calX}\scal{L}{x}\scal{uv}{x}=0
\end{eqnarray*}
meaning that $L$ is an infinitesimal $\conc$-character.

To end, letting $x_i\in X$ and using \eqref{Dshuffle}, one has
\begin{eqnarray*}
&&\scal{(\alpha_0x_0+\alpha_1x_1)^*\shuffle(\beta_0x_0+\beta_1x_1)^*}{x_i}\cr
&&=\scal{(\alpha_0x_0+\alpha_1x_1)^*\otimes(\beta_0x_0+\beta_1x_1)^*}{\Delta_{\shuffle}x_i}\cr
&&=\scal{(\alpha_0x_0+\alpha_1x_1)^*\otimes(\beta_0x_0+\beta_1x_1)^*}{x_i\otimes1_{X^*}+1_{X^*}\otimes x_i}\cr
&&=\alpha_i+\beta_i\cr
&&=\scal{((\alpha_0+\beta_0)x_0+(\alpha_1+\beta_1)x_1)^*}{x_i}.
\end{eqnarray*}
Similarly, letting $y_t\in Y$ and using \eqref{Dstuffle}, one also has
\begin{eqnarray*}
&&\scal{\Bigl(\sum_{s\ge1}a_sy_s\Bigr)^*\stuffle\Bigl(\sum_{s\ge1}b_sy_s\Bigr)^*}{y_t}\cr
&&=\scal{\Bigl(\sum_{s\ge1}a_sy_s\Bigr)^*\otimes\Bigl(\sum_{s\ge1}b_sy_s\Bigr)^*}{\Delta_{\stuffle}y_t}\cr
&&=\scal{\Bigl(\sum_{s\ge1}a_sy_s\Bigr)^*\otimes\Bigl(\sum_{s\ge1}b_sy_s\Bigr)^*}
{y_t\otimes1_{Y^*}+1_{Y^*}\otimes y_t+\sum_{r,s\ge1,r+s=t}y_s\otimes y_r}\cr
&&=a_t+b_t+\sum_{r,s\ge1,r+s=t}a_sb_r\cr
&&=\scal{\Bigl(\sum_{s\ge1}(a_s+b_s)y_s+\sum_{r,s\ge1}a_sb_ry_{s+r}\Bigr)^*}{y_t}.
\end{eqnarray*}
\end{proof}

\begin{example}[Identity $(-t^2y_2)^*\stuffle(t^2y_2)^*=(-4t^4y_4)^*$, \cite{words03,orlando}]\\
$\begin{array}{c}
\begin{tikzpicture}[->,>=stealth',shorten >=1pt,auto,node distance=4cm,
                    semithick,every node/.style={fill=white}]				
  \node[initial,state,accepting] (A)              {$1$};  
	\node[initial,state,accepting] (B) [right of=A] {$2$};
	\node[initial,state,accepting] (C) [right of=B] {$3$};
  \path (A) edge [loop,above] node {$y_2,-t^2$} (A)
        (B) edge [loop,above] node {$y_2, t^2$} (B)
        (C) edge [loop,above] node {$y_4,-t^4$} (C);
\end{tikzpicture}\cr
\begin{array}{r}(-t^2y_2)^*\leftrightarrow(\nu_2,\mu_2(y_2),\eta_2)\phantom{,}\\[-5pt]=(1,-t^2,1),\end{array}
\begin{array}{r}( t^2y_2)^*\leftrightarrow(\nu_1,\mu_1(y_2),\eta_1)\phantom{,}\\[-5pt]=(1,t^2,1),\end{array}
\begin{array}{r}(-t^4y_4)^*\leftrightarrow(\nu,\mu(y_4),\eta)\phantom{.}\\[-5pt]=(1,-t^4,1).\end{array}
\end{array}$
\end{example}

\begin{remark}\label{reetheorem}
In Corollary \ref{Kleene}, if $A=K$ being a field, the points \eqref{caracter} and \eqref{log}
can be rephrased as, respectively, ``$R$ is a group like element" and ``$L$ is a primitive element"
of ${\ncs{K^{\mathrm{rat}}}\calX}$, for $\Delta_{\conc}$.
Indeed, in \eqref{D1}--\eqref{D3}, if $S\in\ncs{K}{Y}$ (resp. $\ncs{K}{\calX}$ is a
$\stuffle$ (resp. $\shuffle,\conc$)-character of  $(\ncp{K}{Y},\conc,1_{Y^*})$
(resp. $(\ncp{K}{\calX},\conc,1_{\calX^*})$ then
\begin{enumerate}
\item Using the fact that $S\otimes S=\sum_{u,v\in\calX^*}\scal{S}{u}\scal{S}{v}u\otimes v$, one has
$\Delta_{\stuffle}S=S\otimes S$ (resp. $\Delta_{\shuffle}S=S\otimes S,\Delta_{\conc}S=S\otimes S$)
meaning that $S$ is group like, for $\Delta_{\stuffle}$ (resp. $\Delta_{\shuffle},\Delta_{\conc}$).
\item Using the fact that $\Delta_{\stuffle}$ (resp. $\Delta_{\shuffle},\Delta_{\conc}$)
and the maps $T\longmapsto T\otimes1_{Y^*},T\longmapsto1_{Y^*}\otimes T$
(resp. $T\longmapsto T\otimes1_{\calX^*},T\longmapsto1_{\calX^*}\otimes T$)
are continuous homomorphisms, one has\footnote{Here, $\log S\otimes1_{Y^*}$ and $1_{Y^*}\otimes\log S$
(resp. $\log S\otimes1_{\calX^*}$ and $1_{\calX^*}\otimes\log S$) commute.}
$\Delta_{\stuffle}\log S=\log S\otimes1_{\calX^*}+1_{\calX^*}\otimes\log S$
(resp. $\Delta_{\shuffle}\log S=\log S\otimes1_{\calX^*}+1_{\calX^*}\otimes\log S,
\Delta_{\conc}\log S=\log S\otimes1_{\calX^*}+1_{\calX^*}\otimes\log S$)
meaning that $\log S$ is primitive, for $\Delta_{\stuffle}$ 
(resp. $\Delta_{\shuffle},\Delta_{\conc}$).
\end{enumerate}
Hence, for $\Delta_{\stuffle},\Delta_{\shuffle}$ and $\Delta_{\conc}$, $S$ is
group like iff $\log S$ is primitive meaning that the equivalence, between
\eqref{caracter} and \eqref{log}, is an extension of the Ree's theorem\footnote{
This theorem was first established for shuffle algebra \cite{ree} and
then was adapted to quasi-shuffle algebra \cite{acta,VJM}.} \cite{ree}.
\end{remark}

\subsection{Exchangeable rational series and their linear representations}\label{TST2}
As examples, one can consider the following forms $(F_0),(F_1)$ and $(F_2)$ of rational
series in $\serie{A^{\mathrm{rat}}}{X}$ \cite{IMACS,CM} and examine their linear representations:
\begin{eqnarray*}
(F_0)&E_1x_{i_1}\ldots E_jx_{i_j}E_{j+1},\mbox{ where }x_{i_1},\ldots,x_{i_j}\in X,E_1,\ldots,E_j\in\serie{A^{\mathrm{rat}}}{x_0},\\
(F_1)&E_1x_{i_1}\ldots E_jx_{i_j}E_{j+1},\mbox{ where }x_{i_1},\ldots,x_{i_j}\in X,E_1,\ldots,E_j\in\serie{A^{\mathrm{rat}}}{x_1},\\
(F_2)&E_1x_{i_1}\ldots E_jx_{i_j}E_{j+1},\mbox{ where }x_{i_1},\ldots,x_{i_j}\in X,E_1,\ldots,E_j\in\serie{A_{\mathrm{exc}}^{\mathrm{rat}}}{X}.
\end{eqnarray*}

\begin{theorem}[Triangular sub bialgebras of $(\ncs{A^{\mathrm{rat}}}{\calX},\shuffle,1_{X^*},\Delta_{\conc},{\tt e})$, \cite{CM}]\label{Subalgebras}
Let $\rho=(\nu,\mu,\eta)$ a representation of $R\in\ncs{A^{\mathrm{rat}}}{\calX}$. Then
\begin{enumerate}
\item\label{comm1} If the matrices $\{\mu(x)\}_{x\in \calX}$ commute between themselves
and if the alphabet is finite, every rational exchangeable series decomposes as 
\begin{eqnarray*}
R=\sum_{i=1}^n \shuffle_{x\in \calX}R_x^{(i)}&\mbox{with}&R_x^{(i)}\in\ncs{A^{\mathrm{rat}}}{x}.
\end{eqnarray*}
\item If $\cal L$ consists of upper-triangular matrices then 
$R\in{\ncs{A_{\mathrm{exc}}^{\mathrm{rat}}}{\calX}}\shuffle\ncp{A}{\calX}$.
\item For any $x\in\calX$, letting $M(x):=\mu(x)x$ and then extending,
in the obvious way, this representation to $\ncs{A}{\calX}$ by 
$M(S)=\sum_{w\in\calX^*}\scal{S}{w}\mu(w)w$, we have
\begin{eqnarray*}
R=\nu{M(\calX^*)}\eta.
\end{eqnarray*}
Moreover, we have
\begin{enumerate}
\item If $\{\mu(x)\}_{x\in\calX}$ are upper-triangular then
$M(\calX)=D(\calX)+N(\calX)$, where $D(\calX)$
and $N(\calX)$ are diagonal and strictly upper-triangular
letter matrices, respectively, such that\footnote{
by Lazard factorization \cite{lothaire,viennotgerard}.}
$M(\calX^*)=((D(\calX^*)N(\calX))^*D(\calX^*))$.
\item We get\footnote{\it idem.} (for $\calX=X$)
\begin{eqnarray*}
M((x_0+x_1)^*)=(M(x_1^*)M(x_0))^*M(x_1^*)=(M(x_0^*)M(x_1))^*M(x_0^*)
\end{eqnarray*}  
and the modules generated by the families $(F_0),(F_1)$ and $(F_2)$
are closed by $\conc,\shuffle$ (and coproducts if $A=K$ is a field).
From this,  it follows that $R$ is a linear combination of expressions
in the form $(F_0)$ (resp. $(F_1)$) if $M(x_1^*)M(x_0)$
(resp. $M(x_0^*)M(x_1)$) is strictly upper-triangular. 
\item If $A$ is a $\Q$-algebra then
\begin{eqnarray*}
M(\calX^*)=\prod_{l\in\Lyn\calX}^{\searrow}e^{S_l\mu(P_ l)}.
\end{eqnarray*}
\end{enumerate}
\end{enumerate}
\end{theorem}

\begin{remark}\label{contrex1}
\begin{enumerate}
\item The point \eqref{comm1} of Theorem \ref{Subalgebras} is no longer true for an infinite
alphabet as shows the example of the series $S=\sum_{k\ge1}y_k$ in $\ncs{A^{\mathrm{rat}}}{Y}$.
\item On a general ring it can happen that $R$ is exchangeable, $\rho$ minimal and nevertheless 
${\cal L}$ is noncommutative, as shows the case of $A=\Q[x,t]/t^3\Q[x,t]$ and
\begin{eqnarray*}
X=\{a,b\},\
\mu(a)=t\begin{pmatrix}1&0\cr x&1\end{pmatrix},\
\mu(b)=t\begin{pmatrix}1& x\cr 0&1\end{pmatrix},\
\nu=\begin{pmatrix}1&1\end{pmatrix},\
\eta=\begin{pmatrix}1\cr1\end{pmatrix}.
\end{eqnarray*}
With these data, $R=2+(xt+2t)(a+b)+(x^2t^2+2xt^2+2t^2)(ab+ba)$
which is an exchangeable polynomial but
\begin{eqnarray*}
\mu(a)\mu(b)=\begin{pmatrix}t^2&xt^2\cr xt^2&x^2t^2+t^2\end{pmatrix},
&\mu(b)\mu(a)=\begin{pmatrix}x^2t^2+t^2& x t^2\cr xt^2&t^2\end{pmatrix}.
\end{eqnarray*}
Now the representation is minimal because if it were of dimension $1$, $\frac{1}{2}R$ would be a 
conc-character, which is not the case. Otherwise, if it were of dimension $0$, $R$ would be zero. 
\end{enumerate}
\end{remark}

In order to establish Theorem \ref{Subbialgebras} below, we will use the following
\begin{lemma}\label{lemma}
Let $(\nu,\tau,\eta)$ a representation of $S$ of dimension $r$ such that, for all 
$x\in \calX$, ($\tau(x)-c(x)I_r$) is strictly upper triangular, then 
$S\in \ncs{K_{\mathrm{exc}}^{\mathrm{rat}}}{\calX}\shuffle\ncp{K}{\calX}$.
\end{lemma}

\begin{proof} Let $(e_i)_{1\le i\leq r}$ be the canonical basis of $K^{1\times r}$.
We construct the representations $\rho_1=(\nu,(x\longmapsto\tau(x)-c(x)I_r),\eta)$, 
$\rho_2=(e_1,(x\longmapsto c(x)I_r),e_1^*)$ of $S_1$ and $S_2$ and remark that
$S_1\shuffle S_2$ admits the representation 
\begin{eqnarray*}
\rho_3=(\nu\otimes e_1,((\tau(x)-c(x)I_r)\otimes I_r+I_r\otimes c(x)I_r)_{x\in \calX},
\eta \otimes e_1^*)
\end{eqnarray*}
as $I_r\otimes c(x)I_r=c(x)I_r\otimes I_r$, $\rho_3$ is, in fact, 
$(\nu\otimes e_1,(\tau(x)\otimes I_r)_{x\in \calX},\eta\otimes e_1^*)$
which represents $S$, the result now comes from the fact that $S_1\in\ncp{K}{\calX}$
and $S_2=(\sum_{x\in \calX}c(x)x)^*\in\ncs{K_{\mathrm{exc}}^{\mathrm{rat}}}{\calX}$.
\end{proof}   

We first begin by properties essentially true over algebraically closed fields. 

\begin{theorem}[Triangular sub bialgebras of 
$(\ncs{K^{\mathrm{rat}}}{\calX},\shuffle,1_{X^*},\Delta_{\conc},{\tt e})$, \cite{CM}]\label{Subbialgebras}
We suppose that $K$ is an algebraically closed field and that $\rho=(\nu,\mu,\eta)$
is a linear representation of $R\in\ncs{K^{\mathrm{rat}}}{\calX}$ of minimal dimension $n$,
we note $\calL=\calL(\mu)\subset K^{n\times n}$ the Lie algebra generated by the matrices 
$(\mu(x))_{x\in\calX}$. Then 
\begin{enumerate}
\item\label{CommRep} $\calL$ is commutative iff $R\in \ncs{K_{\mathrm{exc}}^{\mathrm{rat}}}{\calX}$,
\item $\cal L$ is nilpotent iff 
$R\in{\ncs{K_{\mathrm{exc}}^{\mathrm{rat}}}{\calX}}\shuffle\ncp{K}{\calX}$,
\item $\cal L$ is solvable iff $R$ is a linear combination of expressions in the form $(F_2)$.
\end{enumerate}

Moreover, denoting $\ncs{K_{\mathrm{nil}}^{\mathrm{rat}}}{\calX}$ (resp. $\ncs{K_{\mathrm{sol}}^{\mathrm{rat}}}{\calX}$),
the set of rational series such that  $\calL(\mu)$ is nilpotent (resp. solvable), we get a tower of sub Hopf algebras of
the Sweedler's dual, $\ncs{K_{\mathrm{nil}}^{\mathrm{rat}}}{\calX}\subset\ncs{K_{\mathrm{sol}}^{\mathrm{rat}}}{\calX}
\subset\calH_{\shuffle}^{\circ}(\calX)$.
\end{theorem}
\begin{proof}
\begin{enumerate}
\item Let us remark that, for $x,y\in\calX,p,s\in \calX^*$, we have 
$\scal{R}{pxys}=\scal{R}{pyxs}$ which is due to the commutation of matrices.
Conversely, since $\rho$ is minimal then there is $P_i,Q_i\in\ncp{K}{\calX},i=1...n$
such that (see \cite{berstel,DR,schutz1})
\begin{eqnarray*}
\forall u\in\calX^*,&&
\mu(u)=(\scal{P_i\trr R\trl Q_i}{u})_{1\leq i,j\leq n}=(\scal{R}{Q_iuP_i})_{1\leq i,j\leq n}.
\end{eqnarray*}
Now, for $x,y\in \calX$, we have 
\begin{eqnarray*}
\mu(xy)=(\scal{R}{Q_ixyP_i})_{1\leq i,j\leq n}\stackrel{*}{=}(\scal{R}{Q_iyxP_i})_{1\leq i,j\leq n}=\mu(yx)
\end{eqnarray*}
equality $\stackrel{*}{=}$ being due to exchangeability.

\item Let us consider $K^n$ as the space of the representation of $\calL$ given by $\mu$. 
Let $K^n=\bigoplus_{j=1}^m V_j$ be a decomposition of $K^n$ into indecomposable
$\calL$-modules (see \cite{dixmier}, Theorem 1.3.19 where it is done for $ch(K)=0$,
or \cite{Lie7} Chapter VII \S 1 Propopsition 9 for arbitrary characteristic),
we know that each $V_j$ is a $\calL$-module and that the action of $\calL$
is triangularizable with constant diagonals inside each sector $V_j$.
Thus, it is an invertible matrix  $P\in\mathrm{GL}(n,K)$ such that
\begin{eqnarray*}
\forall x\in \calX,&&P\mu(x)P^{-1}=\mathrm{blockdiag}(T_1,T_2\ldots,T_k)=
\begin{pmatrix}
T_1&0&0&\ldots&0\cr
0&T_2&0&\ldots&0\cr
\vdots&\ddots&\ddots&\ddots&\vdots\cr
0&0&\ldots&0&T_k
\end{pmatrix}
\end{eqnarray*}  
where the $T_j$ are upper triangular matrices with scalar diagonal \textit{i.e.} is of the form
$T_j(x)=\lambda(x)I+N(x)$ where $N(x)$ is strictly upper-triangular\footnote{Even, as $K$ is infinite,
there is a global linear form on $\calL$, $\lambda_{lin}$ such that, for all $g\in\calL$,
$PgP^{-1}-\lambda_{lin}(g)I$ is strictly upper-triangular.}. Set $d_j$ to be the dimension of $T_j$
(so that $n=\sum_{j=1}^m\,d_j$), partitioning  $\nu P^{-1}=\nu'$ (resp. $P\eta=\eta'$)
with these dimensions we get blocks so that each $(\nu'_j,T_j,\eta'_j)$ is the representation
of a series $R_j$ and $R=\sum_{j=1}^m\, R_j$. It suffices then to prove that, for all $j$, 
$R_j\in \ncs{K_{\mathrm{exc}}^{\mathrm{rat}}}{\calX}\shuffle\ncp{K}{\calX}$.
This is a consequence of Lemma \ref{lemma}.

Conversely, if $\rho_i=(\nu_i,\tau_i,\eta_i),i=1,2$, are two representations then
$[\tau_1(x)\otimes I_r+I_r\otimes\tau_2(x),\tau_1(y)\otimes I_r+I_r\otimes\tau_2(y)]
=[\tau_1(x),\tau_1(y)]\otimes I_r+I_r\otimes [\tau_2(x),\tau_2(y)]$
and a similar formula holds for $m$-fold brackets (Dynkin combs), so that if $\calL(\tau_i)$'s
are nilpotent, the Lie algebra $\calL(\tau_1\otimes I_r+I_r\otimes \tau_2)$ is also nilpotent.
The point here comes from the fact that series in $\ncs{K_{\mathrm{exc}}^{\mathrm{rat}}}{\calX}$
as well as in $\ncp{K}{\calX}$ admit nilpotent representations, so, let $(\alpha,\tau,\beta)$
such a representation and $(\alpha',\tau',\beta')$ its minimal quotient (obtained by minimization,
see \cite{berstel}), then $\calL(\tau')$ is nilpotent as a quotient of $\calL(\tau)$.
Now two minimal representations being isomorphic, $\calL(\mu)$ is isomorphic to $\calL(\tau)$
and then it is nilpotent.

\item As $\calL$ is solvable and $K$ algebraically closed, using Lie's theorem,
we can find a conjugate form of $\rho=(\nu,\mu,\eta)$ such that the matrices $\mu(x)$ are upper-triangular.
Since this form also represents $R$, letting $D(\calX)$ (resp. $N(\calX)$) be the diagonal
(rep. strictly upper-triangular) letter matrice such that $M(\calX)=D(\calX)+N(\calX)$ then 
\begin{eqnarray*}
R=\nu M(\calX^*)\eta=\nu(D(\calX^*)N(\calX))^*D(\calX^*)\eta.
\end{eqnarray*}  
Since $D(\calX^*)N(\calX)$ being nilpotent of order $n$ then
$(D(\calX^*)N(\calX))^*=\sum_{j=0}^n(D(\calX^*)N(\calX))^j$.
Hence, letting ${\cal S}$ be the vector space generated by forms of type $(F_2)$ which is
closed by concatenation, we have $D(\calX^*)N(\calX)\in{\cal S}^{n\times n}$ and then
$(D(\calX^*)N(\calX))^*\in{\cal S}^{n\times n}$.
Finally, $R=\nu M(\calX^*)\eta\in{\cal S}$ which is the claim.

Conversely, as sums and quotients of solvable representations 
are solvable is suffices to show that a single form of type $F_2$ admits a solvable representation
and end by quotient and isomorphism as in (ii). From Proposition \eqref{linearrepresentation},
we get the fact that, if $R_i$ admit solvable representations so does $R_1R_2$,
then the claim follows from the fact that, firstly, single letters admit solvable (even nilpotent)
representations and secondly series of $\shuffle\{\ncs{K^{\mathrm{rat}}}{x}\}_{x\in\calX}$
admit solvable representations. Finally, we choose (or construct) a solvable representation of $R$,
call it $(\alpha,\tau,\beta)$ and  $(\alpha',\tau',\beta')$ its minimal quotient,
then $\calL(\tau')$ is solvable as a quotient of $\calL(\tau)$. Now two minimal representations
being isomorphic, $\calL(\mu)$ is isomorphic to $\calL(\tau)$, hence solvable.

Moreover and ff.]
Comes from the computation of the coproduct by insertion of identity $\sum_{i=1}^n\, e_i^*e_i$.
\end{enumerate}
\end{proof}

\begin{remark}
For an example of series $S$ with solvable representation but such that 
$S\notin\ncs{K_{\mathrm{exc}}^{\mathrm{rat}}}{\calX}\shuffle\ncp{K}{\calX}$.
One can take $\calX=\{a,b\}$ and $S=a^*b(-a)^*$.
\end{remark}

To end this section (of combinatorial framework), for a need of the proof of
Theorem \ref{PeanoBacker} below, let us extend the pairing \eqref{pairing}
as a partially defined map  
\begin{eqnarray}\label{pairing2}
\Dom(\sscal{\bullet}{\bullet})
&\longrightarrow&A,\\
T\otimes S&\longrightarrow&\sscal{T}{S}:=\sum_{w\in\calX^*}\scal{T}{w}\scal{S}{w}.
\end{eqnarray}
where $\Dom(\sscal{\bullet}{\bullet})\subset \ncs{A}{\calX}\otimes\ncs{A}{\calX}$.

Here, the family $\sum_{w\in\calX^*}{\scal{T}{w}}\scal{S}{w}$ is summable, for some topology on $A$.
Its sum is denoted by $\sscal{T}{S}$ and the set of these series $S$ is denoted by $\Dom_{word}(T)$.

\section{Towards a noncommutative Picard-Vessiot theory}\label{PVNC}
Let $(\calA,d)$ be a commutative associative differential ring ($\ker(d)=k$ being a field),
$\calC_0$ be a differential subring of $\calA$ ($d(\calC_0)\subset\calC_0$) which is an integral
domain containing the field of constants and $\dext{\C}{(g_i)_{i\in I}}$ be the differential
subalgebra of $\calA$ generated by $(g_i)_{i\in I}$, \textit{i.e.} the $k$-algebra generated by
$g_i$'s and their derivatives \cite{VdP}.

\subsection{Noncommutative differential equations}\label{PV}
Let us consider the following differential equation, with homogeneous series
of degree $1$ as multiplier (a polynomial in the case of finite alphabet).

\begin{eqnarray}\label{NCDE_abs}
\dd S=MS;&\scal{S}{1}=1,&\mbox{where }M=\sum_{x\in\calX}u_x x\in\ncs{\calC_0}{\calX}
\end{eqnarray}

\begin{example}
$X=\{x_0,x_1\}$ and $\Omega=\C\setminus(]-\infty,0]\cup[1,+\infty[)$,
\begin{eqnarray*}
\dd S=(x_0u_{x_0}+x_1u_{x_1})S&\mbox{with}&u_{x_0}(z)=z^{-1},u_{x_1}(z)=(1-z)^{-1}.
\end{eqnarray*}
Solution of this equation is one factor of the solution of the equation $(KZ_3)$ proposed in \cite{drinfeld1,drinfeld2}.
A complete study was presented in \cite{CM} (solutions via polylogarithms and their special values: polyzetas).
\end{example}

\begin{example}
$Y=\{y_i\}_{i\ge1}$ and $\Omega=\{z\in\C\,|\,\abs{z}<1\}$.
\begin{eqnarray*}
\dd S=\biggl(\sum_{i\ge1}y_iu_{y_i}\biggr)S&\mbox{with}&u_{y_i}(z)=\partial\ell_i(z).
\end{eqnarray*}
where, denoting $\gamma$ the Euler's constant and $\zeta$ the Riemann zeta function,
\begin{eqnarray*}
\ell_1(z):=\gamma z-\sum_{k\ge2}\zeta(k)\frac{(-z)^k}k&\mbox{and for}&
r\ge2,\ell_r(z):=-\sum_{k\ge1}\zeta(kr)\frac{(-z^r)^k}k.
\end{eqnarray*}
This equation was introduced in \cite{eulerianfunctions} to study the independence of a family of eulerian functions.
\end{example}

Let us also recall the following useful result for proving Theorem \ref{PeanoBacker} below.

\begin{proposition}[\cite{words03,orlando,acta}]\label{tau}
Let $S\in\ncs{\calA}{\calX}$ be solution of \eqref{NCDE_abs}.
Then $S$ satisfies the differential equations ${\bf d}^lS=Q_lS$, for $l\ge0$,
where $\Q_l\in\ncp{\dext{\C}{(u_i)_{i\ge0}}}{\calX}$ satisfying the recursion
$Q_0=1$ and $Q_l=Q_{l-1}M+{\bf d}Q_{l-1}$.
More explicitly, $Q_l$ can be computed as follows (suming over words $w=x_{i_1}\ldots x_{i_l}$
and derivation multi-indices ${\bf r}=(r_1,\ldots,r_l)$ of degree $\deg{\bf r}=\abs{w}=l$
and of weight ${\tt wgt\;}{\bf r}=l+r_1+\ldots+r_l$)
\begin{eqnarray*}\label{Q_l}
Q_l=\sum_{{\tt wgt\;}{\bf r}=l\atop w\in\calX^{\deg{\bf r}}}
\prod_{l=1}^{\deg{\bf r}}{\sum_{j=1}^lr_j+j-1\choose r_l}\tau_{{\bf r}}(w)
&\mbox{and}&\left\{\begin{array}{r}
\tau_{{\bf r}}(w)=\tau_{r_1}(x_{i_1})\ldots\tau_{r_l}(x_{i_l})=\cr
(\partial_z^{r_1}u_{x_{i_1}})x_{i_1}\ldots(\partial_z^{r_l}u_{x_{i_l}})x_{i_l}.
\end{array}\right.
\end{eqnarray*}
\end{proposition}

\begin{theorem}\label{ind_lin}
Suppose that the $\C$-commutative ring $\calA$ is without zero divisors and equipped with
a differential operator $\partial$ such that $\C=\ker\partial$.

Let $S\in\ncs{\calA}{\calX}$ be a grouplike solution of \eqref{NCDE_abs}, in the following form
\begin{eqnarray*}
S=1_{\calX^*}+\sum_{w\in\calX^*\cal X}\scal{S}{w}w
=1_{\calX^*}+\sum_{w\in\calX^*\cal X}\scal{S}{S_w}P_w
=\prod_{l\in\Lyn\calX}^{\searrow}e^{\scal{S}{S_l}P_l}.
\end{eqnarray*}
Then
\begin{enumerate}
\item If $H\in\ncs{\calA}{\calX}$ is another grouplike solution of \eqref{NCDE_abs}
then there exists $C\in\ncs{\calL ie_{\calA}}{\calX}$ such that $S=He^C$ (and conversely).
\item The following assertions are equivalent
\begin{enumerate}
\item\label{item1} $\{\scal{S}{w}\}_{w\in\calX^*}$ is $\calC_0$-linearly independent,
\item\label{item2} $\{\scal{S}{l}\}_{l\in\Lyn\calX}$ is $\calC_0$-algebraically independent,
\item\label{item3} $\{\scal{S}{x}\}_{x\in\calX}$ is $\calC_0$-algebraically independent,
\item\label{item4} $\{\scal{S}{x}\}_{x\in\calX\cup\{1_{\calX^*}\}}$ is $\calC_0$-linearly independent,
\item\label{item5} The family $\{u_x\}_{x\in\calX}$ is such that, for $f\in\mathrm{Frac}(\calC_0)$
and $(c_x)_{x\in\calX}\in\C^{(\calX)}$,
\begin{eqnarray*}
\sum_{x\in\calX}c_xu_x=\partial f&\Longrightarrow&(\forall x\in\calX)(c_x=0).
\end{eqnarray*}
\item\label{item6} The family $(u_x)_{x\in\calX}$ is free over $\C$ and
$\partial\mathrm{Frac}(\calC_0)\cap\span_{\C}\{u_x\}_{x\in\calX}=\{0\}$.
\end{enumerate}
\end{enumerate}
\end{theorem}

\begin{proof}[Sketch]
The first item has been treated in \cite{orlando}. The second is a grouplike
version of the abstract form of Theorem 1 of \cite{Linz}. It goes as follows
\begin{itemize}
\item due to the fact that $\calA$ is without zero divisors, we have the following embeddings
$\calC_0\subset\mathrm{Frac}(\calC_0)\subset\mathrm{Frac}(\calA)$, $\mathrm{Frac}(\calA)$ is
a differential field, and its derivation can still be denoted by $\partial$ as it induces the
previous one on $\calA$,
\item the same holds for $\ncs{\calA}{\calX}\subset\ncs{\mathrm{Frac}(\calA)}{\calX}$ and ${\bf d}$
\item therefore, equation \eqref{NCDE_abs} can be transported in $\ncs{\mathrm{Frac}(\calA)}{\calX}$
and $M$ satisfies the same condition as previously. 
\item Equivalence between \ref{item1}-\ref{item4} comes from the fact that $\calC_0$ is without zero divisors
and then, by denominator chasing, linear independances w.r.t $\calC_0$ and $\mathrm{Frac}(\calC_0)$ are equivalent.
In particular, supposing condition \ref{item4}, the family $\{\scal{S}{x}\}_{x\in\calX\cup\{1_{\calX^*}\}}$
(basic triangle) is  $\mathrm{Frac}(\calC_0)$-linearly independent which imply, by the Theorem 1 of \cite{Linz},
condition \ref{item5},
\item still by Theorem 1 of \cite{Linz}, \ref{item5} is equivalent to \ref{item6}, implying that
$\{\scal{S}{w}\}_{w\in\calX^*}$ is $\mathrm{Frac}(\calC_0)$-linearly independent which induces
$\calC_0$-linear independence (\textit{i.e.} \ref{item1}).
\end{itemize}  
\end{proof}

Now, let us go back to notations of Section \ref{introduction} and equip the differential rings of
\begin{enumerate}
\item holomorphic functions over a simply connected domain $\Omega$, $(\calH(\Omega),\partial)$,
with the topology of compact convergence (CC),
\item formal series over $\calX$ and with coefficients in $\calH(\Omega)$,
$(\ncs{\calH(\Omega)}{\calX},{\bf d})$, with the ultrametric distance defined by\footnote{
$\forall S\in\ncs{\calH(\Omega)}{\calX}$, if $S=0$ then $\varpi(S)=-\infty$
else $\min_{w\in\mathrm{supp}(S)}\{\abs{w}\mbox{ or }(w)\}$ \cite{berstel}.}
$\delta(S,T)=2^{-\varpi(S-T)}$.
\end{enumerate}
Let us also consider again the Chen series of the differential forms $(\omega_i)_{i\ge1}$
defined by the inputs $\omega_i=u_{x_i}dz$ along a path $z_0\path z$ on $\Omega$.
By \eqref{diagonalX}, it follows that
\begin{eqnarray}
C_{z_0\path z}=\sum_{w\in\calX^*}\alpha_{z_0}^z(w)w=(\alpha_{z_0}^z\otimes\mathrm{Id})\calD_\calX
=\Prod_{l\in\Lyn\calX}^{\searrow}e^{\alpha_{z_0}^z(S_l)P_l}.
\end{eqnarray}
This series satisfies \eqref{NCDE_abs} and is obtained as the limit, for the topology of
(discrete) pointwise convergence over the words, of Picard iteration process initialized at
$\scal{C_{z_0\path z}}{1_{\calX^*}}=1_{\calH(\Omega)}$. 

Let us illustrate Theorem \ref{ind_lin}, with simple examples,
for which $\calC_0$ contains $\dext{\C}{(u_x^{\pm1})_{x\in\calX}}=
\C[u_x^{\pm1},\partial^iu_x]_{i\ge1,x\in\calX}\subset\calA=(\calH(\Omega),\partial)$.
In these examples, we use

\begin{proposition}[\cite{IMACS}]\label{eval_trans}
For $\calX=\{x\}$, since $x^n=x^{\shuffle n}/n!$ then
\begin{eqnarray*}
\alpha_{z_0}^z(x^n)=\frac{(\alpha_{z_0}^z(x))^n}{n!},
&C_{0\path z}=\sum_{n\ge0}\alpha_{z_0}^z(x^n)x^n=e^{\alpha_{z_0}^z(x)x},
&\alpha_0^z(x^*)=e^{\alpha_{z_0}^z(x)}. 
\end{eqnarray*}
\end{proposition}

\begin{example}\label{positive}
Let us consider two positive cases over $\calX=\{x\}$.
\begin{enumerate}
\item $\Omega=\C,u_x(z)=1_{\Omega},\calC_0=\C$.
Since $\alpha_0^z(x^n)=z^n/n!$ then, by Proposition \ref{eval_trans},
\begin{eqnarray*}
C_{0\path z}=e^{zx}&\mbox{and}&\dd C_{0\path z}=xC_{0\path z}.
\end{eqnarray*}
Moreover, $\alpha_0^z(x)=z$ which is transcendent over $\calC_0$ and $\{\alpha_0^z(x^n)\}_{n\ge0}$ is $\calC_0$-free.
Now, let $f\in\calC_0$ then $\partial=0$. Hence, if $\partial f=cu_x$ then $c=0$.
\item $\Omega=\C\setminus]-\infty,0],u_x(z)=z^{-1},\calC_0=\C[z^{\pm1}]\subset\C(z)$.
Since $\alpha_1^z(x^n)=\log^n(z)/n!$ then, by Proposition \ref{eval_trans},
\begin{eqnarray*}
C_{1\path z}=z^x&\mbox{and}&\dd C_{1\path z}=z^{-1}xC_{1\path z}.
\end{eqnarray*}
Moreover, $\alpha_1^z(x)=\log(z)$ which is transcendent over $\C(z)$ then over $\calC_0$
and $\{\alpha_1^z(x^n)\}_{n\ge0}$ is $\calC_0$-free.
Now, let $f\in\calC_0$ then $\partial f\in\mathrm{span}_{\C}\{z^{-n}\}_{n\in\Z,n\not=1}$.
Hence, if $\partial f=cu_x$ then $c=0$.
\end{enumerate}
\end{example}

\begin{example}\label{negative}
Let us consider two negative cases over $\calX=\{x\}$.
\begin{enumerate}
\item $\Omega=\C,u_x(z)=e^z,\calC_0=\C[e^{\pm z}]$.
Since $\alpha_0^z(x^n)=(e^z-1)^n/n!$ then, by Proposition \ref{eval_trans},
\begin{eqnarray*}
C_{0\path z}=e^{(e^z-1)x}&\mbox{and}&\dd C_{0\path z}=e^zxC_{0\path z}.
\end{eqnarray*}
Moreover, $\alpha_0^z(x)=e^z-1$ which is not transcendent over $\calC_0$
and $\{\alpha_0^z(x^n)\}_{n\ge0}$ is not $\calC_0$-free.
If $f(z)=ce^z\in\calC_0$ ($c\neq0$) then $\partial f(z)=ce^z=cu_x(z)$.
\item $\Omega=\C\setminus]-\infty,0],u_x(z)=z^a,a\in\C\setminus\Q,
\calC_0=\dext{\C}{z,z^{\pm a}}=\mathrm{span}_{\C}\{z^{ka+l}\}_{k,l\in\Z}$.
Since $\alpha_0^z(x^n)=(a+1)^{-n}z^{n(a+1)}/n!$ then, by Proposition \ref{eval_trans},
\begin{eqnarray*}
C_{0\path z}=e^{(a+1)^{-1}z^{(a+1)}x}&\mbox{and}&\dd C_{0\path z}=z^axC_{0\path z}.
\end{eqnarray*}
Moreover, $\alpha_0^z(x)=z^{a+1}/(a+1)$ which is not transcendent over $\calC_0$
and $\{\alpha_0^z(x^n)\}_{n\ge0}$ is not $\calC_0$-free.
If $f(z)=cz^{a+1}/(a+1)\in\calC_0$ ($c\neq0$) then $\partial f(z)=cz^a=cu_x(z)$.
\end{enumerate}
\end{example}

\subsection{First step of a noncommutative Picard-Vessiot theory}\label{preliminaries}
Let us recall that the vector space of solutions of \eqref{NCDE_abs}
is a free ($\ncs{\C}{\calX}$-right) module of dimension one\footnote{\label{Sols}
In fact, we will see that it is the $\ncs{\C}{\calX}$-right module
$C_{z_0\path z}.\ncs{\C.1_{\calH(\Omega)}}{\calX}$.}
generated by $C_{z_0\path z}$ \cite{orlando}. Hence, by Theorem \ref{ind_lin},
we have common traits with the ordinary case of first order differential equations,
\begin{enumerate}
\item the differential Galois group of \eqref{NCDE_abs} + grouplike is the Hausdorff
group $\{e^{C}\}_{C\in\ncs{\calL ie_{\C.1_{\calH(\Omega)}}}{\calX}}$
(group of characters of $\calH_{\shuffle}(\calX)$).
\item the PV extension related to \eqref{NCDE_abs} is $\serie{\calC}{\calX}(C_{z_0\path z})$,
where $\calC\subset\calA=(\calH(\Omega),\partial)$ such that
$\mathrm{Const}(\serie{\calC}{\calX})=\ker{\bf d}=\serie{\C.1_{\calH(\Omega)}}{\calX}$.
\end{enumerate}

The proof of Theorem \ref{PeanoBacker} below will use the following lemma
as a consequence of Theorem \ref{residual}
\begin{lemma}\label{FracRank}
For any ring $A$ without zero divisors, let $R\in\ncs{A^{\mathrm{rat}}}\calX$ of
linear representation $(\nu,\mu,\eta)$ of dimension $n$. Then any family
$\{R\trl P_i\vert P_i\in\ncp{A}{\calX}\}_{i=1\ldots m>n}$ is linearly dependent, \textit{i.e.} there are
$\{\alpha_i\}_{i=1\ldots m}$ in $A$, not all zero,  such that $\sum_{I=1}^m\alpha_i(R\trl P_i)=0$.
\end{lemma}

\begin{theorem}[\cite{words03,orlando,acta}]\label{PeanoBacker}
Let $R\in\ncs{\C.1_{\calH(\Omega)}^{\mathrm{rat}}}{\calX}$. Then, for any path
${z_0\path z}$ over $\Omega$, we have\footnote{Once $(z,z_0)$ is fixed on $\Omega$,
$\Dom_{word}(C_{z_0\path z})$ is the subset of $\ncs{A}{\calX}$ of series $R$
such that $\sum_{n\geq 0}\alpha_{z_0}^z(R_n)$ is convergent for the standard
topology, where $R_n=\sum_{|w=n|}\scal{R}{w}w$ is a homogeneous component
(we need to check that this series is convergent via {\it majoration morphisms}
\cite{words03,orlando,acta}).}
$R\in\Dom_{word}(C_{z_0\path z})$ and the output of \eqref{nonlinear} can be computed by
\begin{eqnarray*}
y(z_0,z)=\alpha_{z_0}^z(R)=\sum_{w\in\calX^*}(\nu\mu(w)\eta)\alpha_{z_0}^z(w)=\sscal{C_{z_0\path z}}{R}.
\end{eqnarray*}

Now, let $N$ be the least integer $n$ such that $y$ satisfies
a (non-trivial) differential equation of order $N$ (with coefficents in $\calC$), the family
$\{\partial y\}_{0\le k\le N-1}$ is $\calC$-linearly independent, \textit{i.e.}
\begin{eqnarray*}
(a_n\partial^N+\ldots+a_1\partial+a_0)y=0,&\mbox{with}&a_N,\ldots,a_0\in\calC.
\end{eqnarray*}
and, from what precedes, we have $N\le n=\mathrm{rk}(R)$.
\end{theorem}
\begin{proof}
Due to this strong convergence condition, we have 
\begin{enumerate}
\item for any $T\in \ncs{\calH(\Omega)}{\calX}$ and $ P\in \ncp{\calH(\Omega)}{\calX}, S\in\Dom_{word}(T)$,
we have $S\in\Dom_{word}(PT)$,$S\trl P\in\Dom_{word}(T)$ and $\sscal{PT}{S}=\sscal{T}{S\trl P}$,

\item from the continuity of $\partial$, for any $T\in \ncs{\calH(\Omega)}{\calX}$ and $S\in\Dom_{word}(T)$,
we have $\partial(\sscal{T}{S})=\sscal{\dd T}{S}+\sscal{T}{\dd S}$.
\end{enumerate}

Now, let $(\nu,\mu,\eta)$ be a representation of $R\in\ncs{\C.1_{\calH(\Omega)}^{\mathrm{rat}}}{\calX}$
of rank $n$. Let us see that the family $(\scal{C_{z_0\path z}}{w}\scal{R}{w})_{w\in\calX^*}$ is summable
in $\calH(\Omega)$. Indeed, since the matrix norm is multiplicative then, for any $w\in\calX^*$ and $B_1>0$,
we have\footnote{We choose a matrix norm (\textit{i.e.} multiplicative) on $\C^{n\times n}$, denoted $\absv{M}$,
and two norms $\absv{\nu}_r,\absv{\eta}_c$ on $\C^{1\times n},\C^{n\times 1}$, respectively, and there is
classically $k_1>0$ such that $\abs{\nu.M.\eta}\leq k_1\absv{\nu}_r\absv{M}\absv{\eta}_c$.}
\begin{eqnarray*}
\absv{\mu(w)}\le B_1^{\abs{w}}&\mbox{and}&\abs{\nu\mu(w)\eta}\leq k_1\absv{\nu}_r\absv{\mu(w)}\absv{\eta}_c.
\end{eqnarray*}
The Chen series $C_{z_0\path z}$ is exponentially bounded from above\footnote{
In the references the bounding is finer and adapted as well to infinite alphabet.},
\textit{i.e.} for all compact $\kappa\subset\Omega$, there is $k_2,B_2>0$ such
that\footnote{For any $f\in\calH(\Omega)$, we denote
$\absv{f}_{\kappa}:=\sup_{s\in{\kappa}}\abs{f(s)}$.} \cite{words03,orlando,acta}
\begin{eqnarray*}
\forall w\in\calX^*,&&\absv{\scal{C_{z_0\path z}}{w}}_{\kappa}\le k_2{B_2^{\abs{w}}}/{\abs{w}!}.
\end{eqnarray*}
Hence, choosing a compact ${\kappa}\subset\Omega$, we obtain 
\begin{eqnarray*}
\sum_{w\in\calX^*}\absv{\scal{C_{z_0\path z}}{w}\scal{R}{w}}_{\kappa}
&\le&\sum_{w\in\calX^*}\absv{\scal{C_{z_0\path z}}{w}}_{\kappa}\abs{\sscal{R}{w}}\\
&\le&\sum_{w\in\calX^*}k_2\frac{B_2^{\abs{w}}}{\abs{w}!}(k_1\absv{\nu}_rB_1^{\abs{w}}\absv{\eta}_c)<+\infty.
\end{eqnarray*}
Since $y=y(z_0,z)=\alpha_{z_0}^z(R)$ and $\partial$ is continuous for (CC) then, by Proposition \ref{tau},
\begin{eqnarray*}
\partial^l y(z_0,z)=\sscal{\dd^lC_{z_0\path z}}{R}&\mbox{and for}&l\le n,\dd^lC_{z_0\path z}=Q_l(z)C_{z_0\path z}
\end{eqnarray*}
and then \cite{words03,orlando,acta}
\begin{eqnarray*}
\partial^l y(z_0,z)=\sscal{Q_l(z)C_{z_0\path z}}{R}=\sscal{C_{z_0\path z}}{R\rg Q_l(z)}.
\end{eqnarray*}
By Lemma \ref{FracRank}, there is $\{a_k\}_{k=0,..,n}$ in $\calC$, not all zero,
such that $\sum_{k=0}^na_k(R\trl Q_k)=0$ yielding the expected result.
This linear independence holds in any module whatever the ring.
\end{proof}

\begin{remark}\label{rk3}
\begin{enumerate}
\item The rational series in Theorem \ref{PeanoBacker} is the generating series of the first order
linear differential system, $\partial q=(u_0\mu(x_0)+\ldots+u_m\mu(x_m))q,y=\nu q$, initialized at $y(z_0)=\eta$.
From \cite{fliess1}, $y(z)=\alpha_{z_0}^z(R)$. The $N$th-order differential equation in Theorem \ref{PeanoBacker}
is then the result, obtained by eliminating the states $\{q_i\}_{i=0,..,m}$ in this system.
\item The converse process is also possible thanks to the \textit{compagnion form}.
\item Analogue results for nonlinear equations can be found in \cite{words03,orlando,acta}.
\end{enumerate}
\end{remark}

\section{Conclusion}
In this work, we proposed a first step to construct a Picard-Vessiot theory for the class of
noncommutative differential equations satisfied by the Chen series $C_{z_0\path z}$ over
the alphabet $\calX=\{x_i\}_{i\ge0}$ (along paths $z_0\path z$ belonging to a simply connected
manifold $\Omega$ and with respect to the differential forms $(u_idz)_{i\ge0}$):
\begin{enumerate}
\item The coefficients of these noncommutative generating series belong to the
differential ring $\dext{\C}{(u_i)_{i\ge0}}\{\scal{C_{z_0\path z}}{w}\}_{w\in\calX^*}$
which is closed by integration with respect to $(u_idz)_{i\ge0}$.

\item The Picard-Vessiot extension of these noncommutative differential equations is defined as the
module $C_{z_0\path z}\C1_{\calH(\Omega)}$ and the Haussdorf group $\{e^C\}_{C\in\ncs{\Lie_{C}}{\calX}}$
plays the r\^ole of differential Galois group associated with this extension.

\item These differential equations were considered as universal differential equations
by many authors \cite{cartier1,deligne,drinfeld1,drinfeld2,CM}. Universality can be seen
by replacing each letter by constant matrices (resp. holomorphic vector field, given in
\eqref{vectorfield}) and then solving a system of linear (resp. nonlinear)
differential equations, given in \eqref{nonlinear}.

\item These solutions are obtained as a pairing between the series $C_{z_0\path z}$ and the generating series
of finite Hankel (resp. Lie-Hankel) rank \cite{fliess2,PV,fliess1,reutenauerrealisation}, for linear
(resp. nonlinear) differential equations explaned by Remark \ref{rk3}.

\item Via rational series (on noncommutative indeterminates and with coefficients in rings)
\cite{berstel,reutenauer} and their non-trivial combinatorial Hopf algebras (or pseudo Hopf algebras)
(Theorems \ref{isomorphy}, \ref{residual}, \ref{exchangeable}, \ref{Subalgebras}
and \ref{Subbialgebras}),
we illustrated this theory, still under construction, with the case of linear differential
equations with singular regular singularities (Theorem \ref{PeanoBacker}) thanks to an 
equation satisfied by the Chen generating series.
\end{enumerate}

This practical study allowed also to treat the noncommutative generating series
of multiindexed polylogarithms and harmonic sums and as well as those of
their special values (polyzetas). In particular, we proved the existence of well
defined infinite sums of these polylogarithms and harmonic sums \cite{eulerianfunctions}
in order to describe solutions of differential equations (Theorem \ref{PeanoBacker}).

\end{document}